\documentclass{amsart}
\usepackage{ObermanLatexPackage}

\newcommand{\dx}{h}
\newcommand{\Dir}{\mathcal{D}}
\newcommand{\LD}{\lambda_{\Dir}}
\newcommand{\LDh}{\lambda_{\Dir^W}^\dx}
\newcommand{\dir}{\mathcal{V}}

\newcommand{\Mmn}{M^{N_1\times N_2}}
\newcommand{\Rmn}{\R^{N_1\times N_2}}

\newcommand{\Sk}{{F^\dx}}
\newcommand{\Grd}{G^\dx}

\newcommand{\Fe}{F^\e}

\begin{document}
\title[PDE for the Rank One Convex Envelope]{A partial differential equation for the rank one convex envelope}
\author{Adam~M. Oberman}
\author{Yuanlong Ruan}
\date{\today}
\begin{abstract}
A Partial Differential Equation (PDE) for the rank one convex envelope is introduced. Existence and uniqueness of viscosity solutions to the PDE is established. Elliptic finite difference schemes are constructed and convergence of finite difference solutions to the viscosity solution of the PDE is proven.  Computational results are presented and laminates are computed from the envelopes.  
Results include the Kohn-Strang example, the classical four gradient example, and an example with eight gradients which produces nontrivial laminates.  \end{abstract}


\maketitle

\section{Introduction}
In this article, we establish a nonlinear elliptic Partial Differential Equation (PDE) for the rank one convex envelope.  The PDE is based on a viscosity solutions formulation of the Legendre-Hadamard condition, \eqref{LegendreHadamard} below,  along with an obstacle problem.  

The rank one convex envelope is a generalized convex envelope which arises in nonconvex vector variational problems.  The study of these problems goes back to Morrey \cite{morrey1952quasi} with extensive work in the 1980s \cite{BallElasticity, kohn1986optimal1, BallJames87, chipot1988equilibrium}.  The field is now well-established, with a number of textbook references available \cite{Dacorogna2, Muller, pedregal1997parametrized}.   

In this article, we derive and prove well-posedness (existence and uniqueness of viscosity solutions) for the PDE for the directional convex envelope.  Uniqueness follows from the comparison principle for viscosity solutions.  Existence of viscosity solutions follows from Perron's method: the solutions are continuous up the boundary of the domain. Some of these results are new even in the case of the usual convex envelope.  

We build a wide stencil elliptic finite difference scheme for the directional convex envelope.    The finite difference schemes have  unique solutions which can be found as the fixed point of an iterative method.  The existence and uniqueness results for the solutions of schemes is also new, even in the special case of convex envelopes.  Convergence of the solutions of the numerical scheme to the directional convex envelope follows by applying the Barles-Souganidis convergence theorem. 

Vector variational problems in the two by two matrix case involve functions from $\R^2$ to $\R^2$.  In this case, the corresponding PDE is for scalar functions defined on $\R^4$.  Numerical examples are computed in four dimensions.  From the approximate rank one convex envelope, we compute the associated laminates, by iteratively expanding the barycenter in rank one directions, using points which lie in the rank one convex hull of the minimal level set.

The viscosity solutions formulation of convex functions was studied in \cite{alvarez1997convex}. A related PDE for the (usual) convex envelope of a scalar valued function was derived in~\cite{ObermanConvexEnvelope}.  The regularity of the solution of the PDE was studied in \cite{oberman2011dirichlet} and \cite{de2015optimal}.   In \cite{ball2000regularity} regularity of the rank one convex envelope is established.


Computations of the rank one convex envelope were performed in \cite{Dolzmann, DolzmannWalkington} in four spatial dimensions, using directional convexification.  A convergence rate for solutions was established in \cite{DolzmannWalkington}.   See also~\cite[Chapter 6]{dolzmann2003variational}.
By increasing the number of directions used, at extra computational cost, the rate of convergence of the algorithm was improved \cite{bartels2004linear}.  Polyconvex envelopes were computed in~\cite{bartels2005reliable}.

A wide stencil elliptic finite difference scheme for the convex envelope was presented in \cite{ObermanEigenvalues} and further studied in \cite{ObermanCENumerics}. 
Laminates were previously computed using  a non-convex optimization method by Aranda and Pedregal~\cite{pedregal2001computation, aranda2001numerical}.  
The directional convex envelope, for the special case of coordinate directions, was studied in \cite{matouvsek1998functional}.  An algorithm for the directional convex envelope of a general direction set in the plane was implemented in \cite{franvek2009computing}, along with a proof that the algorithm terminated in polynomial time.    
  
\subsection{Variational problems and generalized convex envelopes}
In this section, we briefly review how the rank one convex envelopes arise in variational problems. Consider the variational problem for vector valued functions $u: \Omega \subset \R^{N_1} \to \R^{N_2}$,
\bq\label{J}
\min_{u\in \mathcal{A} } J(u) = \int_\Omega G(\grad u(x)) dx, 
\eq
for a suitable  set of  admissible functions defined on the domain $\Omega$, 
along with appropriate boundary conditions.

In  the vector-valued case, which corresponds to $N_2 > 1$, minimizers may not exist without some kind of convexity assumption on $G$.  The correct notion of convexity in this setting is \emph{quasiconvexity} \cite{morrey1952quasi}.
The quasiconvex envelope is defined by taking perturbations of $G$ with gradients of smooth, compactly supported functions,~$\phi$,  
\bq\label{QG}
G^{qc}(M) = \inf_{\phi \in C^\infty_0(\Omega, \R^N)} \frac{1}{\abs{\Omega}} \int_\Omega G(M + \grad \phi) dx.
\eq
Replacing the $G$ in  \eqref{QG}  with,  $G^{qc}$, the quasiconvex envelope of $G$, 
results in a problem for which the minimum is attained, and the minimum is equal to the infimum of the original problem.   While this definition is natural, it is not tractable. 
Two related and more tractable notions of convexity have been introduced,  \emph{rank one convexity} and \emph{polyconvexity}.   Rank one convexity is necessary for quasiconvexity, but not sufficient (at least in dimension $N_1=3$) for quasiconvexity. (The notions coincide with convexity in the scalar-valued case.) 
Rank one convexity arises from restricting the minimization in \eqref{QG} to a smaller class of functions.  
The minimizers are gradient Young measures  which correspond to weak solutions of the relaxed minimization problem for the original energy (which has no classical minimizers).   The rank one minimizers are called laminates.  A visualization of the laminates can be found in \cite{Muller} and in \cite{aranda2001numerical, pedregal2001computation}.  The laminates are represented schematically as graphs, with edges in rank one directions (see the next section and \S\ref{sec:NumResults} below).

\subsection{Convexity and rank one convexity}
The function $f: \Rn \to \R$ is convex if
\begin{equation}\label{ce:defn}
f( \lambda x_1 + (1-\lambda)x_2) \le \lambda f(x_1) + (1-\lambda)f(x_2), 	
\end{equation}
for all $x_1,x_2 \in \Rn$ and $\lambda \in [0,1]$.  The convex envelope of the function $f$, $f^{ce}$,  is defined as 
\[
f^{ce}(x) = \sup \{ v(x) \mid v(y) \le f(y) \text{ for all } y, \quad \text{$v$ is convex } \}.
\]
The convex envelope can be represented (see \cite[Theorem 2.35]{Dacorogna2}) as 
\begin{equation}\label{cerep}
	f^{ce}(x) = \inf  \left \{ \sum_{i=1}^{n+1}  w_i f(x_i) ~\middle |~  x = \sum_{i=1}^{n+1}  w_i x_i  \right \}
\end{equation}
where $\sum_{i=1}^{n+1} w_i = 1$, and each $w_i \ge 0$.

Let  $\Mmn$ be the set of $N_1\times N_2$ matrices. A function $G:\Mmn \to \R$ is rank one convex if
\[
G( \lambda F_1 + (1-\lambda)F_2) \le \lambda G(F_1) + (1-\lambda)G(F_2), 
\]
for all  $F_1, F_2 \in \Mmn$   with $\rank(F_1-F_2)=1$,  and $\lambda \in [0,1]$.   If $G$ is twice differentiable, rank one convexity is equivalent to the Legendre-Hadamard condition
\begin{equation}
	\label{LegendreHadamard}
\frac{d^2 G}{d M^2}(F) \ge 0, 
\quad \text{ for all } F, M \in \Mmn \text{ with } \rank(M) = 1.
\end{equation}

The first representation we give of the rank one convex envelope of $G$, is analogous to \eqref{ce:defn}.
\bq\label{R1CEdefn}
G^{rc}(M) = \sup\{  V(M) \mid  V(Y) \leq G(Y) \text{ for all } Y, \quad V \text{ is rank one convex} \}.
\eq	

A second representation for the rank one convex envelope generalizes~\eqref{cerep}, \cite[Section 6.4]{Dacorogna2}.  Assume that there exists at least one rank convex function below $G$.  Then  
\[
	G^{rc}(F) = \inf \left \{ 
	\sum_{i=1}^l \lambda_i G(F_i) ~\middle |~ (\lambda_i, F_i)_{i=1}^l \text{ an $(H_l)$ sequence  with barycenter $F$}
		\right \}
\]
In this case, we have a much more complicated structure for the class of points with a given barycenter.   It is defined in terms of $(H_l)$ sequences.
\begin{definition}
Given $F\in \Mmn$, if we can write
\begin{equation}
F=\lambda_{1}F_{1}+\lambda_{2}F_{2},
\quad \text{ where } \rank(F_{1}-F_{2})\leqslant1
\label{eq:hl}%
\end{equation}
and $\lambda_{1}+\lambda _{2}=1$, $0<\lambda_{1},\lambda_{2}<1$
then we say $(\lambda_i, F_i)_{i=1}^2$ is an $(H_2)$ sequence with barycenter $F$.
Given an $(H_l)$ sequence with barycenter $F$, inductively define an $(H_{l+1})$ sequence with barycenter $F$ by choosing some $F_j$, $j \in \{1,\dots, l\}$ and building an $(H_2)$ sequence $(\mu_k,G_k)_{k=1}^2$  with barycenter $F_j$.  Then replace the single term $(\lambda_j, F_j)$ with the two terms $(\lambda_k \mu_k, G_1), (\lambda_k \mu_2, G_2)$.
The result (after relabelling) is 
\[
(\lambda_i, F_i)_{i=1}^{l+1} \text{ an $(H_{l+1})$ sequence with barycenter } F = \sum_{i=1}^{l+1}\lambda_{i}F_{i}.
\]
\end{definition}

A constructive method for the rank one convex envelope, is also available~\cite{kohn1986optimal1,kohn1986optimal2}.
Let $G_0 = G$ and define iteratively
\begin{equation}\label{KSiteration}
G_{k+1} = \inf \left \{
\lambda G_k(F_1) + (1-\lambda) G_k(F_2) \mid F = \lambda F_1 + (1-\lambda)F_2, ~\rank(F_1 -F_2) = 1
\right \}.	
\end{equation}
Then the iterations converge to $G^{rc}$.

\subsection{Directional convexity}
We give a definition of directional convexity ($\Dir$-convexity) which recovers: (i) standard convexity when $\Dir = \Rn$, and (ii) rank one convexity when $\Dir$ is  the set of rank one directions (where $\Rn$ is identified with $\Rmn$).

\begin{definition}
The set $\Dir \subset \R^n$ is a direction set if (i) the span of $\Dir$ is the entire space and (ii) $\Dir$ is symmetric: if $d \in \Dir$ then $-d \in \Dir$, (iii) $0 \not \in \Dir$. 
The continuous function $u: \Rn \to \R$ is $\Dir$-convex  (directionally convex) if
\begin{equation} \label{d.convex}
u(\lambda x + (1-\lambda) y) \le \lambda u(x) + (1-\lambda)  u(y), 
\quad \text{ for all $0\le \lambda \le 1$, and all ${x-y} \in  \Dir$}.
\end{equation}
The $\Dir$-convex envelope of a given function $g$ is defined as the pointwise supremum of all $\Dir$-convex functions which are majorized by~$g$,
\bq\label{DCEdefn}
g^{\Dir}(x) = \sup\{  v(x) \mid  v(y) \leq g(y) \text{ for all } y, \quad v \text{ is $\Dir$-convex} \}.
\eq	
\end{definition}

\begin{remark}
In the case where $u$ is  twice differentiable, it can be seen by taking the limit of finite differences, that  $\Dir$-convexity implies 
\bq\label{convexsecondderiv}
\frac{ d^2 u }{ d v^2 }  \ge 0, \quad \text{ for all $v\in \Dir$}.
\eq
which generalizes the Legendre-Hadamard condition~\eqref{LegendreHadamard} to general direction sets.
\end{remark}

\section{The PDE for the rank one convex envelope}
In this section we study the fully nonlinear elliptic Partial Differential Equation
for the directionally convex ($\Dir$-convex) envelope, (problem \eqref{obstacle}, below).  This equation includes the rank one convex envelope and the (usual) convex envelope as special cases.  A synthetic example of a directional convex envelope for a different set of directions is also presented below, in Example~\ref{ex:SyntheticFourGradient}, for illustration.

Comparison results for viscosity solutions are well-established.  The standard comparison result of viscosity solutions theory is \cite[Theorem 3.3]{CIL}, which applies to operators which are either uniformly elliptic or strictly proper.  Neither of these apply to PDE~\eqref{obstacle}.   However, the same result can be applied in the special case where it is possible to perturb a supersolution to a strict supersolution.   This is the strategy applied below, which is described in more detail in the sequel. 
The existence of solutions, and continuity up to the boundary, is established using Perron's method.

\subsection{The $\Dir$-convex envelope operator}
We consider the problem on a bounded domain $\Omega$, with $\Omega \subset D= [0,1]^n$.  Assume that the given function $g:\Rn \to \R$ is continuous, and that there exists a continuous  function $g_0:\Rn \to \R$ with
\begin{equation}\label{g_assumption}
g_0 \text{ is $\Dir$-convex on $\Rn$},
\qquad
g \geq g_0\text{ in } \Omega, 
\qquad
g = g_0 \text{ on } D\setminus \Omega.
\end{equation}

\begin{remark}\label{rem:gquad}
The assumption \eqref{g_assumption} is consistent with previous work, for example Lemma 9.7 of \cite{pedregal1997parametrized} and Theorem 6.10 of \cite{Dacorogna2}.
Often we are interested in values $g\le c_0$.  A natural way to enforce \eqref{g_assumption} in this case is to simply replace $g$ with $\max(g,g_0)$ where $g_0$ is a large quadratic function. 
\end{remark}

\begin{remark}
Viscosity solutions of the Dirichlet problem need not be continuous up to the boundary \cite[Section 7]{CIL}.  Additional assumptions which ensure continuity up to the boundary can be of two types.   The first type is a regularity requirement of the boundary.  For the Laplacian operator, a barrier can be constructed for domains satisfying an exterior cone condition \cite{GTBook}.   For the Dirichlet problem for convex envelope, solutions are continuous up to the boundary if the boundary is strictly convex~\cite{caffarelli1986dirichlet, oberman2011dirichlet}.    On the other hand, on square domains,  if $g$ is concave, the convex envelope may be strictly below $g$ on the boundary \cite{ObermanCENumerics}.   In fact, our computations are usually performed on non-strictly convex domains. 

Continuity up to the boundary is needed to apply the Barles-Souganidis theorem~\cite{BSNum}  (This requirement is referred somewhat confusingly to as \emph{strong comparison} in the article).   However recent work by Froese, \cite{froeseGauss}, establishes convergence away from the boundary  without the strong comparison assumption. 

The assumption \eqref{g_assumption}, which we missed in our earlier work~\cite{ObermanCENumerics}, allows us to establish continuity up to the boundary. 
\end{remark}

\begin{definition}
Let $\Dir$ be a direction set in $\Rn$, and let $\mathcal{S}^{n}$ be the set of symmetric $n \times n$ matrices.  
Define the $\Dir$-convexity operator, $\LD : \mathcal{S}^{n} \to \R$, 
\begin{equation}\label{LDdefn}
\LD(M) =  \inf_{v \in \Dir} \frac{1}{\abs{v}^2} v^\intercal M v, 
\end{equation}
and the $\Dir$-convex envelope operator, $F^{\Dir,g}:  \mathcal{S}^{n} \times \R \times \Rn \to \R$
\bq\label{DCE}
F^{\Dir,g}(M,r,x) = \max \left\{  r - g(x), -\LD(M)   \right\} 
\eq		
When the context is clear, we write $F= F^{\Dir,g}$. 
\end{definition}
The obstacle problem for the $\Dir$-convex envelope in $\Omega$ is to solve
\begin{equation}\label{obstacle}\tag{DCE}
F^{\Dir,g}(D^2u(x),u(x),x) =  \max \left\{  u(x) - g(x), -\LD(D^2u(x))   \right\}  = 0, 
\end{equation}
for $x \in \Omega$, along with Dirichlet boundary conditions 
\begin{equation}\label{BCD}\tag{D}
u(x) = g(x),\quad  \text{for } x \in \partial \Omega	.
\end{equation}

\subsection{Definition of viscosity solutions}
\begin{definition}\label{defn:degenerateelliptic}
	The function $F:  \mathcal{S}^{n} \times \R \times D : \to \R$ is proper and degenerate elliptic (in the sense of \cite{CIL}) if 
\[
F(M,r,x) \le F(N,s,x), \quad \text { for all } M \succeq N, r \le s, \text{ and all } x \in D
\]	
where $Y \preceq  X$ means $d^\intercal Yd \le d^\intercal X d$ for all $d\in \Rn$.
\end{definition}

\begin{lemma}\label{lem:degell}
	The functions $F^{\Dir,g}$ and $-\LD$ are degenerate elliptic, in other words
\[
-\LD(X) \le -\LD(Y),
\quad \text{ whenever $Y \preceq X$ }
\]	
and	
\[
F^{\Dir,g}(X,r,x) \le F^{\Dir,g}(Y,s,x), \quad \text{ whenever $r \le s$ and $Y \preceq X$}.
\]	
furthermore, for any constant $c$, 
\begin{equation}\label{LDhomog}
	\LD(X + cI)  = \LD(X) + c,
\end{equation}

\end{lemma}

\begin{proof}
	First suppose $X \preceq Y$.  Then for all $d \in \Rn$, $d^\intercal X d \le d^\intercal Y d$.  So $\LD(X) \le \LD(Y)$. 
	Next, it is clear from the definition \eqref{DCE} that $F^{\Dir,g}$ is non-decreasing in $r$. 
Combining this with the previous result gives the second assertion of the Lemma. 

Finally we show that \eqref{LDhomog} holds.  Simply compute
\[
\LD(X + cI) = \inf_{v \in \Dir} \frac{1}{\abs{v}^2} v^\intercal (X + cI) v
= \inf_{v \in \Dir} \frac{1}{\abs{v}^2} v^\intercal X v + c 
= \LD(X) + c
\qedhere
\]
\end{proof}

Next we define viscosity solutions of~\eqref{obstacle}. 

\begin{definition}[Upper and Lower Semicontinuity]\label{def:envelope}
Let $u: \Rn \to \R$.   The \emph{upper and lower semicontinuous envelopes} of $u(x)$ are defined, respectively, by
\[ u^*(x) = \limsup_{y\to x}u(y), \]
\[ u_*(x) = \liminf_{y\to x}u(y). \]
The function $u$ is upper semicontinuous, $u \in USC(\Rn)$, if  $u = u^*$, and 
$u$  is lower semicontinuous, $u \in LSC(\Rn)$,  if  $u = u_*$.
\end{definition}

\begin{definition}
The function $u \in USC$ is a viscosity subsolution of $-\LD(D^2 u(x)) = 0$
if for every $C^2$ function $\phi$, whenever $x$ is a local  maximum of $u-\phi$ at $x$,
\bq\label{LD} 
-\LD(D^2\phi(x)) \leq 0.
\eq
The function $u \in USC$ is a viscosity subsolution of~\eqref{obstacle}
if for every $C^2$ function $\phi$, whenever $x$ is a local  maximum of $u-\phi$ at $x$,
\begin{equation}
	\label{obbSub}
u(x) - g(x) \leq 0 \quad \text{ and } \quad -\LD(D^2\phi(x))  \leq 0,  
\end{equation}
The lower semicontinuous function $u$ is a viscosity supersolution of~\eqref{obstacle} if  whenever  $\phi \in C^2$ touches $u$ from below at $x$ 
\begin{equation}\label{obsSuper}
u(x) - g(x) \geq 0 \quad \text{ or } \quad -\LD(D^2\phi(x))  \geq 0.	
\end{equation}
A function $u$ is a viscosity solution of \eqref{obstacle} if it is both a subsolution and a supersolution. 	
\end{definition}

\subsection{Comparison principle for the PDE}
Next we state a technical, but standard, viscosity solutions result, which gives the comparison principle in the case where we have strict sub and supersolutions.

\begin{theorem}[Comparison Principle for strict subsolutions \cite{CIL} ]\label{thm:Comp1}
Consider the Dirichlet problem  for the degenerate elliptic operator $F(M,r,x)$ on the bounded domain $\Omega$.  Let $u \in USC(\bar \Omega)$ be a viscosity subsolution and let $v \in LSC(\bar \Omega)$ be a viscosity supersolution.  Suppose further that for $\e > 0$,
\begin{align*}
	F(D^2u(x),u(x),x) + \e  &\le  0 \text{ in } \Omega
\\
 	F(D^2v(x),v(x),x) &\ge 0 \text{ in } \Omega
\end{align*}
holds in the viscosity sense.  Then the comparison principle holds:
\[
\text{ $u \le v$ on $\partial \Omega$ implies $u \le v$ on $\Omega$ }
\]
\end{theorem}

\begin{remark}
In \cite[Section 5.C]{CIL}, it is explained how the main comparison theorem, \cite[Theorem 3.3]{CIL},
can be applied when it is possible to perturb a subsolution to a strict subsolution.   This version of the theorem is what we state in \autoref{thm:Comp1}.  This result was used in \cite[Theorem 3.1]{Bardi2006709} and  \cite{bardi2013comparison} to prove a comparison principle. 	
\end{remark}

We provide a formal proof of \autoref{thm:Comp1}, which can be made rigorous in the case that one of $u$ or $v$ is $C^2$. It is included to illustrate the connection between the comparison principle and Definition~\ref{defn:degenerateelliptic}.

\begin{proof}[Formal proof of \autoref{thm:Comp1}]
Suppose $u$ and $v$ are $C^2$ functions, and $u \not \le v$ in $\Omega$.  Then $\max_{x\in \Omega} \{u(x) - v(x) \} > 0$.  Let $x \in \argmax_{x\in \Omega} \{ u(x) - v(x) \}$.  Then $x$ is in the interior of $\Omega$, since we assumed $u \le v$ on $\partial \Omega$.

Since $x$ is a positive local maximum of $u-v$, we have
\[
u(x) \ge v(x),
\quad
\nabla u(x) = \nabla v(x), 
\quad
D^2u(x) 
\preceq
D^2 v(x)
\]
Using the inequalities above in Lemma~\ref{lem:degell}, we have
\[
F(D^2u(x),u(x),x) \ge   F(D^2v(x),v(x),x).
\]
This last inequality contradicts the strict inequality in the assumption of the Theorem.  So $u \le v$ in $\Omega$.
\end{proof}

In the next result, we show how to perturb a subsolution to obtain a strict subsolution, allowing us to appeal to \autoref{thm:Comp1} to obtain the comparison result.

\begin{theorem}[Comparison Principle]\label{thm:Comp2}
Consider the Dirichlet problem \eqref{obstacle}, \eqref{BCD} for the $\Dir$-convex envelope on the bounded domain $\Omega$. 
Assume that \eqref{g_assumption} holds.  Let $u \in USC( \Omega)$ be a viscosity subsolution of \eqref{obstacle} and let $v \in LSC(\Omega)$ be a viscosity supersolution of \eqref{obstacle}.
Then the comparison principle holds:
\[
\text{ $u \le v$ on $\partial \Omega$ implies $u \le v$ on $\Omega$. }
\]
\end{theorem}

\begin{proof}
We will show that for small enough $\e > 0$  we can perturb $u$ to a function $u_\e$ so that
\begin{align*}
F(D^2u_\e(x), u_\e(x), x) + \e &\le 0, 
\end{align*}
holds in the viscosity sense for all $x$ in $\Omega$.  Then, since $
F(D^2v(x),v(x),x) \ge 0$, we can  apply \autoref{thm:Comp1} to $u_\e, v$ to obtain $u_\e  \le v$ in  $\Omega$.   Taking $\e \to 0$ gives the desired result.

Set $R = \max_{x\in \Omega} \abs{x}$.  Consider the function 
\[
\phi(x) = \frac{\abs{x}^2 - R^2 - 2}{2}.
\]
Then $\phi(x) \le -1$ in $\Omega$, and 
\[
\grad \phi(x) =  x,
\qquad
D^2\phi(x) =  I,
\]
the identity matrix. 
Given the viscosity subsolution $u\in USC(\bar \Omega)$, let 
\[
u_\e(x) = u(x) + \e \, \phi(x)
\]
Then  
\begin{equation}\label{eq11}\tag{i}
u_\e-g \le u - g - \e. 	
\end{equation}
Since $-\LD(D^2 u(x)) \leq 0$ in $\Omega$ holds in the viscosity sense, 
we also have, using Lemma~\ref{lem:degell}, that
\begin{equation}\label{eq21}\tag{ii}
-\LD(D^2u_\e(x)) 
\leq - \e 
\end{equation}
holds in the viscosity sense in $\Omega$.  Together, \eqref{eq11},\eqref{eq21} imply
\begin{align*}
F(D^2u_\e, u_\e(x), x) &= \max \left \{ u_\e(x) - g(x), -\LD(D^2u_\e(x))  \right \} 
\\
& \le \max  \left \{ u(x) - g(x) - \e, -\LD(D^2 u(x)) - \e \right \}
\\ 
&\le F(D^2u,u,x) - \e.	
\end{align*}
So $F[u_\e] + \e \le F[u] \le 0$ as desired. 
\end{proof}

\subsection{Existence of solutions by Perron's method}

In this section we that viscosity solutions of  \eqref{obstacle} \eqref{BCD} are indeed the $\Dir$-convex envelope of the function $g(x)$, assuming \eqref{g_assumption} holds.

We state a lemma, which generalizes a consistency result for the convex envelope which was first obtained in \cite[Lemma 1]{alvarez1997convex}.  That result was used in \cite{ObermanConvexEnvelope} to derive the obstacle problem for the convex envelope.  
\begin{lemma}\label{thm:char}
The  continuous function 
$u: \R^n \to \R$ is $\Dir$-convex if and only if it is a viscosity solution of $-\LD(D^2 u(x)) \leq 0$.
\end{lemma}
\begin{proof}
	We omit the proof, since it is very similar to previous results; the main modification being a restriction to directions in $\Dir$. 
\end{proof}

Next we paraphrase Perron's method. 
\begin{prop}[Perron's method {\cite[Theorem 4.1]{CIL}}]
Suppose the comparison principle holds for \eqref{obstacle} \eqref{BCD}.  Suppose also that there are a continuous subsolution $u_1$ and a continuous supersolution $u_2$ that satisfy \eqref{BCD}.   Then there exists a solution, $w$, of \eqref{obstacle} \eqref{BCD} which is given by
\begin{equation}\label{Perron}
w(x) = \sup \left \{ w(x) \mid u_1 \leq w \leq u_2 \text{ and $w$ is a subsolution of \eqref{obstacle}\eqref{BCD} } \right \}	
\end{equation}
In particular $w = g$ on $\partial \Omega$. 
\end{prop}

\begin{theorem}\label{thm:PDE} 
Suppose $g$ satisfies \eqref{g_assumption}.  
Then the unique viscosity solution of~\eqref{obstacle}\eqref{BCD} is the the $\Dir$-convex envelope of the function $g$.  Furthermore, it is continuous up to the boundary and satisfies \eqref{BCD}.
\end{theorem}

\begin{proof}
By the definition of viscosity solutions, \eqref{obsSuper}, $g$ is a supersolution of \eqref{obstacle}.  
Since $g_0$ is $\Dir$-convex, by Lemma~\ref{thm:char}, $g_0$ is a viscosity supersolution of \eqref{LD}.
By assumption \eqref{g_assumption} $g_0 \leq g$.  Together, these last two assertions show that $g_0$ is a viscosity subsolution of \eqref{obstacle}.   Also by \eqref{g_assumption}, $g = g_0$ on $\partial \Omega$. 
So we have a sub and super solution which satisfy \eqref{BCD}.

By Theorem~\ref{thm:Comp2}, the comparison principle holds for \eqref{obstacle} \eqref{BCD}. 
So we can apply \eqref{Perron} to obtain the solution $w$ which satisfies \eqref{BCD}.

Next we establish that the solution $w$ is the $\Dir$-convex envelope of $g$.
By Lemma~\ref{thm:char}, subsolutions of~\eqref{obstacle} consist precisely of those $\Dir$-convex functions majorized by~$g$.  

According to  \eqref{DCEdefn}, the $\Dir$-convex envelope, $g^\Dir$, is defined for functions defined on all of $\Rn$. However, by assumption \eqref{g_assumption},  $g^\Dir = g_0$ outside of $\Omega$.  So we can restrict to $\Omega$ and we find that the class of functions in Perron's method \eqref{Perron} is the same as in the definition \eqref{DCEdefn}, so  $w = g^\Dir$.   
\end{proof}

\section{An elliptic finite difference method for the PDE}
In this section we present the numerical method for computing the $\Dir$-convex envelope. We show that there exist unique solutions of the finite difference equations, using discrete versions of the comparison principle, and a fixed point method.  We obtain the formal accuracy of the scheme, and prove convergence.

\begin{remark}[Non grid-aligned directions]
The directional finite difference operator along grid directions given an elliptic and hence convergent method to enforce convexity along grid directions.  We will approximate the direction set $\Dir$ by a directions available on the grid, denoting these directions by $\Dir^W$.  For consistency,  we will need to send both $\dx \to 0$ and $\Dir^W \to \Dir$.  	
 The method of \cite{Dolzmann} also enforced directional convexity on a large but restricted direction set.   The methods of \cite{ObermanCENumerics} and \cite{ObermanEigenvalues}  also used wide stencil to approximate the convex envelope operator using grid directions.   

For non grid-aligned directions, the corresponding directional finite difference operator is not monotone.  In fact, there is no monotone, second order accurate method for approximating the second derivative in a non-grid aligned direction~\cite{MotzkinWasow}.   However, it may be possible to use a filtered scheme~\cite{froese2013convergent} to give a convergent method for non-grid aligned directions. This could be done by writing the vector $d$ as the sum of:  (i) a convex combination of nearby grid directions, which is elliptic, and (ii) a quadratic correction term, which is not elliptic.  By  filtering the second term, we could obtain a convergent scheme for a larger direction set.   However, we limit ourselves to the simpler discretization for the present.
\end{remark}

\subsection{Wide stencil finite differences for the $\Dir$-convex envelope}
As before, we consider $\Omega \subset D = [ -1,1]^n$. 
In order to have second order accurate finite difference operators, we use a uniform grid of spacing, $\dx$,  in $D$
\begin{equation}\label{Gdefn}
\Grd = \{ x \in h \Zb^n  ~\mid ~  x\in D \}, 
\qquad
\Grd_V = \Omega \cap \Grd,
\qquad
\partial \Grd = \Grd \setminus \Grd_V.
\end{equation}

\begin{remark}
	Notice that the boundary grid points, $\partial\Grd$ may contain multiple grid points in each grid direction.\end{remark}
	
\begin{definition}[Finite difference equation]
Let $C(\Grd)$ denote the set of grid functions, $u: \Grd \to \R$. 
A finite difference operator is  map 
 $F^\dx: C(\Grd) \to C(\Grd)$, which has the following form,
\begin{equation}\label{fd.general}
F^\dx[u](x) = F^\dx(x, u(x), u(x) - u(\cdot) ), 	
\end{equation}
where $u(\cdot)$ indicates the values of the grid function $u$.
It has \emph{stencil width} $W$ if 
 $F^\dx(x, u(x), u(x) - u(\cdot) )$ depends only on values $u(y)$ for $\norm{y-x}_\infty/\dx  \le W$.  A \emph{solution} of the finite difference scheme is a grid function which satisfies the equation $F^\dx[u](x) = 0$ for all $x \in \Grd$.
\end{definition}

\begin{example}[Centred second differences]
The centred second difference operator is given by
\[
D^h_{xx}[u](x)  = \frac{1}{\dx^2} 
\left( u(x+h) - 2u(x) + u(x-h)\right).
\]
The operator is consistent, and second order accurate: for smooth functions $u$,  $D^h_{xx}[u](x) = u_{xx}(x) + \bO(h^2)$. 
\end{example}

The centred second difference  operator is naturally extended to second derivatives in a grid direction. 

\begin{definition}[Grid Direction Set]\label{defn:gridDir}
We call $v \in \mathbb{Z}^n$ a grid vector, and define its width to be  $W = \norm{v}_\infty$.  If $\Dir^W$ is a direction set consisting of grid vectors, then the width of the direction set is the maximum width of any grid vector in the set. 
\end{definition}

\begin{definition}[Grid directional second derivatives]
 Suppose that $x \pm \dx v \in \Grd$ for all $x \in \Grd_V$. Define  
the finite difference operator $D^\dx_{vv}: C(\Grd) \to C(\Grd_V)$  by
\begin{equation}\label{DvvDefn}
	D^\dx_{vv}u(x) = \frac{ u(x+\dx v) - 2(u) + u(x- \dx v)} {h^2 \norm{v}^2}, \quad x \in \Grd_V 
\end{equation}
The stencil width  of $D^\dx_{vv}$ is the width of $v$.\end{definition}
Write $\hat v = v/\norm{v}$.  Then for any smooth function $u(x)$, 
\begin{equation}\label{FiniteDiffError}
\frac{d^2 u}{d \widehat{v}^2}(x)  =  D^\dx_{vv}u(x)+ \bO(W^2 h^2)
\end{equation}

\begin{definition}
The direction set of width $W$ for the convex envelope is given by 
\[
\Dir^{CE,W} = \{ v \in \mathbb{Z}^n \mid \norm{v}_\infty \le  W \}.
\]
The direction set of width $W$ for the rank one convex envelope
  is defined  by first setting 
$\dir^{W2}  = \{  v\otimes w   \mid v,w \in\Dir^{CE,W} \}$, which is a set of rank one square  matrices.  Let $vec(M):M^{n\times n}\to \R^{n^2}$ be the natural identification of a matrix with a vector. Set 
\[
\Dir^{RC,W} = \{ vec(M) \mid M \in \dir^{W2}  \}
\]
\end{definition}

An illustration of a typical two dimensional grid, with grid directions indicated by edges can be found in~\autoref{fig:geometricgraphs}.
\begin{figure}
\begin{center}
\scalebox{.6}{\includegraphics{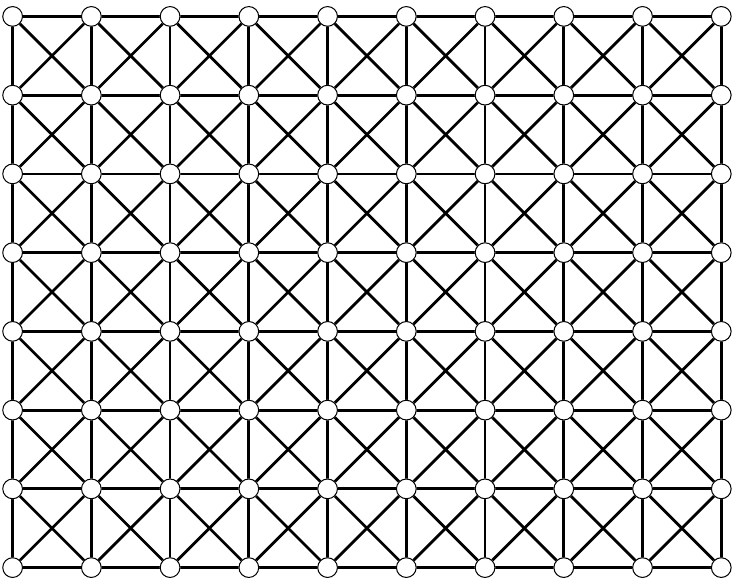}}
\end{center}
\caption{An illustration of a two dimensional grid $\Grd$ with direction set $\Dir^\dx$ indicated by the edges, corresponding to Example~\ref{ex:SyntheticFourGradient}.}
\label{fig:geometricgraphs}
\end{figure}
The full discretization of \eqref{obstacle}, which includes the boundary conditions, is given by the following.
\begin{definition}[Full discretization of PDE]\label{defn:fulldisc}
The discretization  on the finite difference grid $\Grd$ with direction set $\Dir^W$ of the operator $\LD$ (defined by \eqref{LDdefn}), is given by
\begin{equation} \label{LDh}
\LDh[u](x) = \min_{v \in \Dir^W} D^\dx_{vv}[u](x)
\end{equation}
The full discretization of \eqref{obstacle} is given by inserting \eqref{LDh} into \eqref{obstacle}.
\begin{equation}\label{fulldisc}
F^{W,h}[u](x) = \max \left\{  u(x) - g(x), -\LDh[u](x)  \right\}  = 0, \quad x \in \Grd_V
\end{equation}
\begin{equation}\label{GridBC}
	F^{W,h}[u](x) = u(x) - g(x), \quad x \in \partial \Grd
\end{equation}	
\end{definition}

\begin{remark}
We take $\dx$ small enough so that $x\pm h v \in \Grd$ for all $x\in \Grd_V$ and $v\in \Dir^W$. 
\end{remark}

\subsection{Elliptic difference schemes and the discrete comparison principle}
We define elliptic difference schemes in a general setting, and show that the  discretization \eqref{fulldisc} \eqref{GridBC} is elliptic. 
Then we prove that solutions of the discrete equation are unique.  The proof follows the pattern of the proof of uniqueness for the PDE.

\begin{definition}[Elliptic finite difference schemes]
The finite difference operator given by \eqref{fd.general}, $
\Sk[u](x) = \Sk(x, u(x), u(x) - u(\cdot) )$ 
is elliptic if 
\begin{equation}\label{Fellipticscheme}
	r \le s, ~ v(\cdot) \le w(\cdot) \implies 
\Sk(x,r,v(\cdot)) \le \Sk(x,s,w(\cdot))
\end{equation}
\end{definition}

\begin{lemma}\label{lem:SkEll}
	The finite difference operators $-\LDh$ and $F^{W,h}$ given by \eqref{LDh} and \eqref{fulldisc} are degenerate elliptic.
	\end{lemma}

\begin{proof}
It is clear that $D^\dx_{vv}$ is elliptic. 

The finite difference operator $\LDh$ is a nondecreasing function of the directional second derivatives $-u_{vv}$, so it is elliptic.   The operator $F^{W,h}$ is a nondecreasing function of $u(x)$ and $-\LDh$ so it is also elliptic. 
\end{proof}

\begin{definition}[Discrete Comparison Principle]
Given the finite difference operator $\Sk: C(\Grd) \to C(\Grd)$, 
\emph{the comparison principle} holds for $\Sk$ if
\[
\Sk(u)  \le \Sk(v) \implies u \le v.
\]
\end{definition}

\begin{remark}
In the Discrete Comparison principle, the boundary conditions are encoded in $\Sk$: the assumption $\Sk[u] \leq \Sk[v]$ means $u \le v$ at Dirichlet boundary points.  Uniqueness of solutions clearly follows from the Discrete Comparison Principle. 
\end{remark}

\begin{lemma}\label{FellLem}
Suppose the scheme $\Sk$ is elliptic. 
If $x$ is a non-negative global maximizer of $u-\phi$ we have
\[
\Sk(u)(x) \ge \Sk(\phi)(x)
\]
\end{lemma}

\begin{proof}
By assumption, $u(x) \ge \phi(x)$ and 
\[
u(x) - u(y) \ge \phi(x) - \phi(y)
\]
for all values of $y$.   Thus we have
\[
\Sk[u](x) = \Sk(x,u(x),u(x) - u(\cdot)) \ge \Sk(x,\phi(x),\phi(x) - \phi(\cdot)) 
\]
by \eqref{Fellipticscheme} since $\Sk$ is elliptic. 
\end{proof}

\begin{lemma}[Discrete Comparison Principle for strict subsolutions]\label{thm:Comp3}
Let $\Sk$ be an elliptic difference equation on the grid $\Grd$. 
 Let $u,v$ be grid functions.  Suppose  that for some $\e > 0$,
\begin{align*}
	\Sk[u] + \e  \leq \Sk[v] 
\end{align*}
Then the comparison principle holds:
\[
\text{ $u \le v$ on $\Grd$. }
\]
\end{lemma}

\begin{proof}
Suppose $u \not \le v$ in $\Grd$.  Let $x \in \argmax_{x\in \Grd} \{ u(x) - v(x) \}$.  
Then $x$ is a positive global maximum of $u-v$, so by Lemma~\ref{FellLem}
\[
\Sk[u](x) \ge   \Sk[v](x)
\]
which contradicts the assumption of strict inequality in the statement of the Theorem.  So $u \le v$.
\end{proof}

\begin{theorem}\label{thm:DiscreteComparison}
The comparison principle holds on $\Grd$ for the  discrete obstacle problem \eqref{fulldisc}\eqref{GridBC} for the $\Dir^\dx$-convex envelope.
\end{theorem}
The proof follows the pattern of the proof of Theorem~\ref{thm:Comp2}, but in a discrete setting. 
\begin{proof}
We will show that for small enough $\e > 0$  we can perturb $u$ to a grid function $u_\e$ so that
\[
\Sk[u_\e] + \e  \leq \Sk[v] 
\]
Then we apply Lemma~\ref{thm:Comp3} to $u_\e, v$ to obtain $u_\e  \le v$ in  $\Grd$.   Taking $\e \to 0$ gives the desired result.  

We again use a quadratic function, except now it is a grid function. Set 
\[
\phi(x) = \frac{\abs{x}^2 - R^2 - 2}{2}, 
\qquad
u_\e(x) = u(x) + \e \, \phi(x)
\]
where $R = \max_{x\in \Grd} \abs{x}$. 

For any grid vector, $v$, $D^\dx_{vv}[\phi](x) =  1$, and so 
\[
D^\dx_{vv}[u_\e](x) =  D^\dx_{vv}[u](x) + \e 
\]
which means that 
\[
\LDh[u_\e](x) = \min_{v \in \Dir^W}   D^\dx_{vv}[u_\e](x) = \min_{v \in \Dir^W}   D^\dx_{vv}[u](x) + \e = \LDh[u](x) + \e 
\]
Also, $\phi(x) \le -1$, so 
\[
u_\e-g \le u - g - \e. 	
\]
Together the last two inequalities imply
\begin{align*}
\Sk[u_\e](x) &= \max \left \{ u_e(x) - g(x), -\LDh(D^2u_\e(x))  \right \} 
\\
& \le \max  \left \{ u(x) - g(x) - \e, -\LDh(D^2 u(x)) - \e \right \}
\\ 
&\le \Sk[u](x) - \e.	
\qedhere
\end{align*}
\end{proof}

\begin{corollary}\label{lem:stability}
Any solution $u^\dx$ of  \eqref{fulldisc}\eqref{GridBC} is  bounded independently of $\dx$, in particular
\[
\min_{x\in \Grd } g(x) \le u^\dx(x) \le \max_{x\in \Grd } g(x),
\qquad \text{ for all $x\in \Grd$.}
\]
 
\end{corollary}

\begin{proof}
	This result follows from Theorem~\ref{thm:Comp3}, using  the facts that $g(x)$ is a supersolution, and the constant function $m$ is a subsolution. 
\end{proof}

\subsection{Existence of solutions by an iterative method}
In this section we will prove existence of solutions of the finite difference equation  using an iterative method.  The iterative method will also be used numerically to find  solutions of \eqref{fulldisc}\eqref{GridBC}.   

\begin{definition}[Iterative solution method] Define the map, $T: C(\Grd) \to C(\Grd)$  by
\bq\label{Solver_Iterative}
T(u)(x) = \min_{v \in \Dir^W}    \left (  g(x), \frac{u(x + hv)  + u(x- hv) }{2}\right )
\eq
for $x \in \Grd_V$ and $T(u)(x) = g(x)$ for $x\in \partial \Grd$.
\end{definition}

\begin{lemma}
	The grid function $u$ is a solution of \eqref{fulldisc} \eqref{GridBC}  if and only if it is a fixed point of $T$.	
\end{lemma}

\begin{proof}
Let $u$ be a solution of \eqref{fulldisc} \eqref{GridBC}.
For $x\in \partial\Grd$, the fixed point condition and the equation are the same.  So consider $x \in \Grd_V$.  The case $u(x) = g(x)$ is also clear.  So suppose $u(x)< g(x)$.  Multiply the second equation inside the maximum in \eqref{fulldisc} by the factor of $2/\abs{hv}^2$ since the right hand side of the equation is zero.  Then solving for the reference variable $u(x)$ leads to $u(x) = \min_{v \in \Dir^W}\left (g(x), (u(x + hv)  + u(x- hv))/{2}\right )$.  The steps can be reversed to show that a fixed point is a solution. 
\end{proof}

\begin{lemma} \label{lemma:existence}
There exist solutions to the finite difference equations \eqref{fulldisc} \eqref{GridBC} which are fixed points of  \eqref{Solver_Iterative}.
\end{lemma}

We make use of the Brouwer fixed point theorem: a continuous function from a convex, compact subset $K$ of Euclidean space to itself has a fixed point.  This fixed point will be the solution of the equation.

\begin{proof}
Identify $C(\Grd)$ with $\R^N$, where $N$ is the number of grid points in $\Grd$.
Set 
\[
m = \min_{x\in \Grd} g(x), \qquad
M =  \max_{x\in \Grd} g(x).
\]
and define the convex, compact set $K \subset \R^N$, 
\[
K = 
\left\{ 
u \in C(\Grd) ~\middle |~
\begin{aligned}
u(x) =  g(x),&&  x\in \partial \Grd,
\\
m \le u(x) \le g(x),&&  x \in \Grd_V
\end{aligned}
\right \},
\]
We need to show that  $T(K) \subset K$. If $u \in K$, then  $m \le u(x) \le M$ and for all $x$.  Then   
\[
m \le \frac{u(x + v)  + u(x- v) }{2} \le M, \quad \text{ for all $x \in \Grd_V$ and all $v \in \Dir^W$}.
\]
Since by definition, $m \le g(x) \le M$,
The last result implies that 
\[
m \le T(u)(x) \le M,
\qquad \text{ for all } x,
\]
which in turn means that $T(K)\subset K$.
\end{proof}

\subsection{Accuracy and consistency}

\begin{definition}[Consistent]\label{def:consistentSimpler}
The scheme~$\Sk$ is \emph{consistent} with the continuous function~$F$, if for any smooth function $\phi$ and $x\in{\Omega}$,
\[
\lim_{h \to 0,y\to x} \Sk[\phi](y) = F(D^2 \phi(x), \grad \phi(x),\phi(x),x)
\]
\end{definition}

\begin{definition}
Let $\Dir$ be a direction set and $\Dir^W \subset \Dir$ a grid direction set. 
The directional resolution of $\Dir^W$ (with respect to $\Dir$) is largest angle between any vector in $\Dir$ and the best approximation of it in $\Dir^W$ 
\begin{align}
\label{dtheta}
d\theta &\equiv \max_{ w \in \Dir} \min_{v \in \Dir^W }  \cos^{-1}(w^\intercal v).
\end{align}
\end{definition}

The following directional  estimate is used to establish a consistency result for approximations of directional convex functions using the smaller grid  direction sets. 
Consistency of the full discretization follows. 

\begin{lemma}[Consistency]\label{lem:cons}
Let $\Dir$ be a direction set and $\Dir^W \subset \Dir$ a grid direction set with directional resolution  $d\theta < \pi/2$. For any smooth function $u:\Rn\to \R$, 
\begin{equation}\label{consist}
\LDh u(x) - \LD u(x)  = \bO \left( (W\dx)^2 + d\theta  \right)	
\end{equation}
\end{lemma}

\begin{proof}
Choose $w$ so that $\LD u(x) = \frac{d^2 u}{dw^2}(x)$ and $\norm{w} = 1$.  (If the infimum in $\LD$ is not a minimum, approximate it to within $\e$ by the value at $w$, and send $\e$ to zero). 
Compute
\[
\LDh u(x) = \min_{z\in \Dir^W} \frac{d^2 u}{dz^2}(x) \ge \frac{d^2 u}{dw^2}(x) = \LD u(x), 
\]
since $\Dir^W \subset \Dir$.
Let $\tilde v$ be a vector in $\Dir^W$ whose direction is closest to $w$. Write $v = \tilde v/\norm{\tilde v}$.  Let $\theta$ be the angle between $w$ and $v$. By~\eqref{dtheta},  $\theta \leq d\theta$.
Decompose 
\[
v = \cos\theta  w + \sin\theta z
\]
where $z$ is a unit vector orthogonal to $w$. Then compute
\[
\frac{d^2 u }{d v^2} 
 = \cos^2\theta \frac{d^2u}{d w^2} + \sin^2\theta \frac{d^2u}{dz^2} + 2 (\sin\theta \cos\theta) \,  w^\intercal D^2u\, z
\]
which gives 
\[
\frac{d^2 u }{d v^2}  = \frac{d^2u}{d w^2} + \bO(d\theta).
\]
Next, let  $g$ be a grid vector in the direction $v$.  The error for the finite difference expression for the second derivative from \eqref{FiniteDiffError} gives the additional term $\bO(\norm{g}^2) = \bO( (W h)^2)$,
\end{proof}

\subsection{Convergence}
We first paraphrase the Barles-Souganidis convergence theorem.  We include a proof for the convenience of the reader.  Our proof is slightly simpler than the original proof, because we assume our schemes are elliptic instead of monotone and stable.  As shows above, it is often easy to show that solutions of elliptic schemes are uniformly bounded, which satisfies the stability requirement of the theorem.  On the other hand, monotone schemes need not be stable~\cite{ObermanSINUM}. 
%
%

\begin{theorem}[Convergence of Approximation Schemes \cite{BSNum} ]\label{thm:converge}
Consider the Dirichlet problem for the elliptic PDE, $F[u]=0$, \eqref{BCD},  on the bounded domain $\Omega$. 
Suppose the Comparison Principle holds.  For each $\e>0$,  let $u^\e$ be the solution of the consistent, elliptic finite difference scheme $F^\e$.
Assume that 
\begin{equation}\label{ueDir}
u^\e \in C(\bar\Omega), \qquad
u^\e  = g \text{ on } \partial \Omega	
\end{equation}
and that the functions $u^\e$ are bounded uniformly in $\e$. 
Then 
\[
u^\e \to u, \quad \text{ uniformly on $\bar \Omega$  as } \e \to 0.
\]
\end{theorem}

Before proving the theorem, we state a standard lemma.
\begin{lemma}[Stability of Maxima]\label{lem:sequences}
Let $\Omega \subset \Rn$ be a domain, and let $u^\e \in USC (\bar \Omega)$  
be uniformly bounded.
Define
\[ 
\bar{u}(x) = \limsup_{\e\to0,y\to x} u^\e(y) 
\]
Suppose $x_0$ is the unique  global maximizer  of $\bar{u}$, with $\bar u(x_0) \ge 0$. 
Then there exist sequences $\e_n \to 0$, $y_n \to x_0$ such that  
\[
\begin{cases}
u^{\e_n}(y_n) \to \bar{u}(x_0)
\\
 y_n \text{ is a non-negative global maximum of }  u^{\e_n}.
\end{cases}
\]
\end{lemma}
\begin{proof}
This is standard technical result from the theory of viscosity solutions.  A proof can be found, in, for example, \cite[Lemma 2]{froese2013convergent}.
\end{proof}

\begin{proof}[Proof of Theorem~\ref{thm:converge}]
Define
\[ 
\bar{u}(x) = \limsup_{\e\to0,y\to x} u^\e(y),
\qquad
\underline{u}(x) = \liminf_{\e\to0,y\to x} u^\e(y). 
\]
Then $\bar{u}(x) \in USC(\bar{\Omega})$, $\underline{u}(x) \in LSC(\bar{\Omega}).$  Clearly from the definition,
\[
\underline{u} \leq \bar{u} \quad \text{ in }  \bar\Omega.
\]
By assumption \eqref{ueDir}, $\bar{u} = \underline{u}$ on $\partial \Omega$.
If we know that $\bar{u}$ is a subsolution and $\underline{u}$ is a supersolution, then we could apply the Comparison Principle, Theorem~\ref{thm:Comp2}, to $\bar{u}$ and $\underline{u}$ to conclude that
\[ 
\bar{u} \leq \underline{u}  \quad \text{ in }  \bar\Omega
\]
Together the last two inequalities imply that $\bar{u} = \underline{u}$, and that the limit $u$ is continuous.   Uniform convergence follows.

 It remains to show that $\bar{u}$ is a subsolution and $\underline{u}$ is a supersolution.    Given a smooth test function $\phi$, let $x_0$ be a strict global maximum of $\bar{u}-\phi$ with $\phi(x_0) = \bar{u}(x_0)$.  (We can assume that a local maximum is global by perturbing the test function.) 

By Lemma~\ref{lem:sequences}, applied to $v^\e = u^\e -\phi$,  we can find sequences with 
$\e_n \to 0, y_n \to x_0$, 
$u^{\e_n}(y_n) \to \bar{u}(x_0)$, 
where $y_n$ is a non-negative  global maximizer of $u^{\e_n}-\phi$. Then
\begin{align*}
0 &= F^{\e_n}[u^{\e_n}](y_n) 
&& \text{ since $u^{\e_n}$ is a solution},
  \\
  &\geq F^{\e_n}[\phi](y_n)
  && \text{ by Lemma~\ref{FellLem}, since $\Fe$ is elliptic.}
\end{align*}
Next, 
\begin{align*}
0 &\geq \liminf_{n\to\infty} F^{\e_n}[\phi](y_n) \geq \liminf_{\e\to0,y\to x}F^{\e}[\phi](y)  
\\  &= F(x_0, \phi(x_0),\nabla\phi(x_0),D^2\phi(x_0)),
&& \text{by consistency of $F^\e$} 
  \\
  &= F(x_0, \bar{u}(x_0),\nabla\phi(x_0),D^2\phi(x_0)),
&& \text{since $\bar u(x_0) = \phi(x_0)$ }
\end{align*}
which shows that $\bar{u}$ is a subsolution.  

By a similar argument,  we can show that $\underline{u}$ is a supersolution.  
\end{proof}
 
Next we apply the convergence theorem in our setting.
We need to show that: our schemes are consistent, our schemes are elliptic, solutions to the schemes exist and are uniformly bounded, and that the PDE is well posed. Finally, we need to know the boundary conditions hold for $u^\e$ and $u$ in the strong sense. 

\begin{remark}[Interpolating the grid functions]
The numerical solutions are given on a grid, but to apply the theorem we need continuous  functions defined on $\Omega$.  To achieve this, simply fix a triangulation of the domain, and use piecewise linear interpolation of the grid functions.  To be precise, we would need to consider the full solution operator which includes the linear interpolation.  However, since the interpolation does not affect the necessary properties of the scheme, we can safely neglect this detail. 
\end{remark}

\begin{theorem}\label{thm:convergence}
Let $u = g^\Dir$ be the $\Dir$-convex envelope of $g$,  and suppose \eqref{g_assumption} holds.  
Let $u^{W,\dx}$ be the solution of the elliptic finite difference equation   $F^{W,\dx}$  \eqref{fulldisc} \eqref{GridBC}, and let $d\theta$ be the directional resolution.  Then 
\[
u^{W,\dx} \to u \text{ uniformly,} 
\quad
\text{ as } 
\dx, d\theta \to 0,
\]
\end{theorem}

\begin{proof}
We first show that \eqref{ueDir} holds.  First note that $g$ is a supersolution of $F^{W,\dx}$.   Next, since by assumption~\eqref{g_assumption}, $g_0$ is $\Dir$-convex, and since $\Dir^W$ is a subset of $\Dir$, this implies that $g_0$ is $\Dir^W$ convex.  So $g_0$ is a subsolution of $F^{W,\dx}$.   By the discrete Comparison Principle, Theorem~\ref{thm:DiscreteComparison}, $g_0 \leq u^{W,\dx} \leq g$, so in particular, $u^{W,\dx} = g$ on $\partial \Omega$.  

By Theorem~\ref{thm:PDE} the $u = g^\Dir$ is the unique viscosity solution of  the Dirichlet problem for the $\Dir$-convex envelope \eqref{obstacle}\eqref{BCD}.  

The scheme $F^{W,\dx}$ is elliptic by Lemma~\ref{lem:SkEll}.  It is consistent by Lemma~\ref{lem:cons}.  Solutions of the scheme exist by Lemma~\ref{lemma:existence}.   
The functions $\bar{u}, \underline u$ are bounded between $m = \min_x g(x)$ and $M = \max_{x} g(x)$, by Lemma~\ref{lemma:existence} (or by Lemma~\ref{lem:stability}).

Combining these results, we can apply Theorem~\ref{thm:converge}.
\end{proof}

\section{Algorithm for finding laminates from the rank one convex envelope}\label{sec:algo}

We will approximate infinite order laminates by growing trees on the graph determined by the grid and the grid direction set.   
A similar algorithm is described in \cite{dolzmann2003variational}.  The trees need not terminate, but each branch of the tree eventually terminates in an extreme point.  Since each time a new branch is created, the corresponding weights decrease geometrically, we can approximate an infinite order laminate by a finite tree.

Define a $\Dir^W$ tree on the grid $\Grd$ to be an   $(H_l)$ sequence which lies on 
 $\Grd = \Grd_V \cup \partial \Grd$ and uses the direction set $\Dir^W$.

\begin{remark}[Visualization of Laminates]
A visualization of the construction of laminates can be found in \autoref{fig:lam2d} below. The cross denotes the barycenter, the hollow circle denotes the points resulting from each decomposition, the solid circle denotes the supporting points. The order refers to the number of decompositions involved. The same conventions apply to all figures that follow.
\end{remark}

\begin{definition}
Consider a graph with vertices $K \subset \Grd_V$, and whose edges $(x,y)$ consist of those pairs where $y-x$ is in the direction of some $d \in \Dir$.
A path in $K$ is a sequence $x_1, \dots, x_n$ where $x_i\in K$ and $x_{i+1}-x_i$ are edges. 
\end{definition}
 
\begin{definition}
A $\Dir$-tree in $K$ is given recursively by the following.  The single vertex $x$ is a $\Dir$-tree with root $x$.   Given any $\Dir$-tree, and any vertex with degree 1 or less, we can add the vertices $x_+$ and $x_-$ if both $x_-,x_+ \in K$ and 
 \[
x_+ - x = k_+ d, 
\qquad
x_- - x = k_- d, 
\quad \text{ for } k_- < 0 < k_+, \text{ and some  }  d \in \Dir. 
\]
	
\end{definition}

\begin{definition}[$\Dir$-extreme points, $\Dir$-boundary points]
For $x\in K$, and $d\in \Dir$ we say $K$ is $d$-connected at $x$ if both $x+d$ and $x-d$ are in $K$. 
We say $x\in D$ is 
\begin{align*}
&\text{ an interior point,}  \text{ if $K$ is $d$-connected at $x$ for all } d \in \Dir
\\
&\text{ a boundary point,}   \text{ if $K$ is $d$-connected for some but not all $d\in \Dir$ at $x$ }
\\
&\text{ an extreme point,}  \text{ if  $K$ is not $d$-connected at $x$ for any }  d \in \Dir
\end{align*}
Partition $\Dir = \Dir^+ \cup \Dir^-$ where for each $d\in \Dir$ exactly one of $d, -d$ is in each of $\Dir^+$, $\Dir^-$ and choose an ordering $d_1^+,d_2^+, \dots, d_n^+$ for $\Dir^+$, and a corresponding ordering for $\Dir^-$. 
\end{definition}

\begin{lemma}\label{lem:extremepoints}
	Let $K \subset \Grd_V$  If $K$ is nonempty, then $K$ contains an extreme point.
For each $x\in K$, there is a finite path (branch of the tree) with directions in $\Dir^+$ which terminates at an extreme point.
\end{lemma}
\begin{proof}
	Given $x\in K$, if $x$ is not extremal, choose a path in $K$ which does the following: move as far as possible in the direction $d$, choosing $d$ from $\Dir^+$, with order of priority given by the ordering.   Since $D^+$ introduces a partial ordering on $K$ and each point in the path is comparable under the ordering, it is impossible to return to a previous point.  Since the set is finite, the path must terminate.  At the terminal point, it is impossible to move in any direction in $\Dir^+$.  So the terminal point is an extreme point of $K$.	
\end{proof}

Given the grid function $g \in C(\Grd)$, let $u = g^\Dir$ be the $\Dir$-convex envelope of $g$, and let 
 \[
K = \{ x \in \Grd_V \mid  u(x) = m = \min_{y\in \Grd} g(y) \}
 \]
be the minimal level set of $u(x)$ which is assumed to lie in $\Grd_V$ (recall from Corollary~\ref{lem:stability} that $u(x) \geq m$).   Define the set of supporting points
\[
P = \{ x \in K \mid u(x) = g(x)\}
\]  
Notice that every extreme point $x$ of $K$ is a supporting point.  (Suppose not, then $u(x) < g(x)$, so $\LDh u(x) = 0$, which means $u(x) = (u(x+d) + u(x-d))/2$ for some $d \in \Dir^W$, which contradicts the fact that $x$ is extreme.)

Our algorithm for extracting laminates from $K$, corresponds to decomposing a point $x\in K$ into an $(H_l)$ sequence, or $\Dir$-tree. 

Use the fixed ordering of the direction set, $\Dir$, and fix the maximum number of recursions, $N_{L}$.
Given $x \in K$, if $x$ is an extreme point, terminate.  If not, choose a direction $d \in \Dir$  in order of priority: 
\begin{enumerate}
\item Choose a direction which allows $x$ to be decomposed into two extreme points 
\item Choose a direction which allows $x$ to be decomposed into one extreme point, and one boundary point.
\item Otherwise, decompose $x$ into two boundary points, both distinct from $x$, choosing from directions where  $E$ is $d$-connected at $x$, according to the ordering. (Notice that this is possible even if $x$ is a boundary point).
\end{enumerate}
Extend $x$ in two directions, as far as possible, to the points $x+k_1d, x-k_2 d$.  Record the corresponding weights for each of the endpoints.  Apply the algorithm recursively to both endpoints, stopping at extremal endpoints or when the recursion limit is reached.

\section{Numerical Results}\label{sec:NumResults}
In this section we present computations of $\Dir$-convex envelopes, and laminates. We also present solution times, and convergence results.

\begin{remark}[Values of the parameters in practice]
 In practice, in four dimensions, we use grids with less than $100$ points in each dimension, or about $25$ million variables, and we use at most $256$ grid directions, which corresponds to stencils of width $W=3$.  We test convergence of the method in both parameters $\dx, d\theta$. 
 
We first wrote the code in MATLAB, where the largest examples took a few hours.  We then implemented the solver in $C$, which improved the solution time to under 10 minutes for the Kohn-Strang example with the largest grid size using $256$ directions.  Other examples took longer, see the numerical results section below.  See also Remark~\ref{rem:speed} for further improvements to solution time. 
\end{remark}

This algorithm \eqref{KSiteration} was implemented in \cite{Dolzmann} and studied in \cite{DolzmannWalkington}.  A quantitative error estimate for the difference between the rank one convex envelope, and the numerical directional convex envelope using a finite number of directions on a grid of resolution $h$ was established in \cite{DolzmannWalkington}.  	The directions which are used are given by 
\[
\Dir^h = \{ a \otimes b \mid \abs{a}, \abs{b} \le h^{-1/3} \}
\]
in that case, the convergence rate
\[
\abs{G^h - G^{rc}}_\infty \le C \abs{G}_{Lip} h^{1/3}
\]
is established.  

\begin{remark}
In practice, in \cite{DolzmannWalkington} the smallest value of $\dx$ used was $1/65$ which corresponds to $65^4$ variables. In our case, using a laptop we had  a grid of size $71^4$ which corresponds to  $\approx 25\times 10^6$ variables.
So in the convergence rate estimate, $h^{1/3} = .2$.  Certainly we are outside the asymptotic regime.
\end{remark}

\begin{remark}\label{rem:solution_methods_PDE}
An alternative to the iterative method is to perform one dimensional  directional convex envelopes (for which we have fast algorithms) and iterate these over the directions.   In two spatial dimensions, convexification along lines was faster that the iterative method.  But when the direction set is large, for example using $256$ directions in the four dimensional case, convexification along lines is much slower than the iterative method.  The solution times for the two methods are presented in Table~\ref{table:solutiontimes2d}.
 \end{remark}

\begin{remark}[Improved solution speed]\label{rem:speed}
	After this article was completed, we discovered a method to find solutions much more quickly. We found in~\cite{ObermanQuasicConvex} that by iterating a line solver (for a different type of envelope) with a moderate number of iterations of the iterative solver, we could significantly improve the solution speed.   Instead of taking on the order of $1/\dx^2$ iterations of the iterative solver, we could alternately perform (i) a line solver for each direction and (ii) $1/\dx$ iterations of the iterative solver.  Doing this about 10 times resulted in the solution to within a small tolerance.  Experiments with convex envelopes obtained comparable results.  We expect similar results for this problem. 	
\end{remark}

\begin{example}[Specific choices of direction sets]
We label the following direction sets, which are used in building the direction set for the computational examples. 
\begin{align*}
\dir_{4} &= ~~~~~~\left\{  (1,0), (0,1), (-1,1), (1,1) \right \} \\
\dir_{8} &= \dir_{4}  \cup  \{(2,1), (1,2), (-1,2), (-2,1) \} \\
\dir_{16} &= \dir_{8} \cup \{ (3,1), (3,2), (2,3), (1,3), (-3,1), (-3,2), (-2,3), (-1,3) \}
\end{align*}
We define the following rank one direction sets, which correspond to width one, two, and three stencils.
\begin{align*}
\Dir_{16} = \dx \dir_{4}\otimes\dir_{4},  
\qquad
\Dir_{64} = \dx \dir_{8}\otimes\dir_{8},   
\qquad
\Dir_{256} = \dx \dir_{16}\otimes\dir_{16} 
\end{align*}	
\end{example}

\subsection{The Kohn-Strang example}

\begin{example}[The Kohn-Strang example]
In this section we consider the example from \cite{kohn1986optimal1} \cite{kohn1986optimal2}.  The accuracy of solutions we found was quite similar to the values reported in~\cite{Dolzmann},  

The computation used
\[
G(M) =
\begin{cases}
	1 + \abs{M}^2, & M \not = 0
	\\
	0, & M = 0
\end{cases}
\]
The rank one convex envelope is given by
\[
G^{rc}(M) =
\begin{cases}
	1 + \abs{M}^2, & \rho(M) \ge 1
	\\
	2\rho(M) - 2D, & \rho(M) \le 1
\end{cases}
\]
where $D = \abs{\det{M}}$ and $\rho(M) = \sqrt{\abs{M}^2 + 2D}$.
This calculation is for a discontinuous function $G$.  
 Another option is to consider (as in \cite{Dolzmann})
\[
\tilde G(M) =
\begin{cases}
	1 + \abs{M}^2, & \abs{M} \ge \sqrt 2 - 1
	\\
	2\sqrt 2 \abs{M} & \text{ otherwise} \end{cases}
\]	
In this case, we show the error (which is the same) and computation times in Table~\ref{table:ContinuousKohnTime}.  The computation times were longer for this example.
\end{example}

We computed both examples, and found the error was the same.  The longest computational time for the first example was 10 minutes, compared to about half an hour for the second example. 
In Table~\ref{table:ContinuousKohnTime} we also present the error in the maximum norm, and the computational time.  Note that the error is dominated by the $h$, improving $d\theta$ does not improve the error.  This is not the case for later (less symmetric) examples.

\begin{table}[ptb]
\centering
\begin{tabular}
[c]{ccllll}%
Gidsize & $dx$ & $\Dir_{16}$ & $\Dir_{64}$ & $\Dir_{144}$ & $\Dir_{256}$ \\\hline
$45^4$ & 0.2500 & 0.0439 (  3.8) & 0.0439 ( 9.34) & 0.0439 (   7.25) & 0.0439 (  13)\\
$57^4$ & 0.1667 & 0.0385 ( 23.9) & 0.0278 ( 67.7) & 0.0278 (   69.6) & 0.0278 ( 136)\\
$69^4$ & 0.1250 & 0.0672 ( 95.5) & 0.0313 (290.8) & 0.0313 ( 363.5) & 0.0313 ( 693)\\
$81^4$ & 0.1000 & 0.0760 (282.6) & 0.0139 (906.6) & 0.0139 (1218. ) & 0.0139 (2218)
\end{tabular}
\caption{Computational error and time (in seconds) for the Kohn-Strang example smoothed at the origin.
}%
\label{table:ContinuousKohnTime}
\end{table}

\subsection{The Classical Four Gradient Example}
We begin with a classical example, which is discussed in \cite[Section 2.5]{Muller}.  It is also referred to as the 
\begin{example}\label{ex:ClassicalFourGradient}
\label{eg4gradient}

Consider the set  $K=\{A_{1},A_{2},A_{3},A_{4}\}$,  of four $2\times2$ diagonal matrices, 
\begin{equation}\label{K4g}
A_{1}=-A_{3}=\left(
\begin{array}
[c]{cc}%
-1 & 0\\
0 & -3
\end{array}
\right)  ,\text{ }A_{2}=-A_{4}=\left(
\begin{array}
[c]{cc}%
-3 & 0\\
0 & 1
\end{array}
\right)  .
\end{equation}

There are no rank one connections in the set $K$. The rank one convex hull of $K$ is the unit square plus four segments connecting the four supporting points.  In this example, since all four matrices are diagonal, it reduces to a two-dimensional problem.
Here the rank one convex hull can also be regarded as the $D_2$-convex hull where
\[
\Dir_{2} = \{ e_{1}, e_{2} \},
\qquad
e_{1}=\left(
\begin{array}
[c]{cc}%
1 & 0\\
0 & 0
\end{array}
\right)  ,e_{2}=\left(
\begin{array}
[c]{cc}%
0 & 0\\
0 & 1
\end{array}
\right).
\]
This allows the computation to be performed in two dimensions.  See Figure~\ref{fig:4Ghull} for the 
envelope and hulls, and for an illustration of the laminates which are extracted
directly from the computed rank one convex hull of $K.$   Different orderings of the direction sets can give different laminates.

\end{example}

\begin{figure}
\scalebox{.37}{
\hspace{-3cm}
\includegraphics{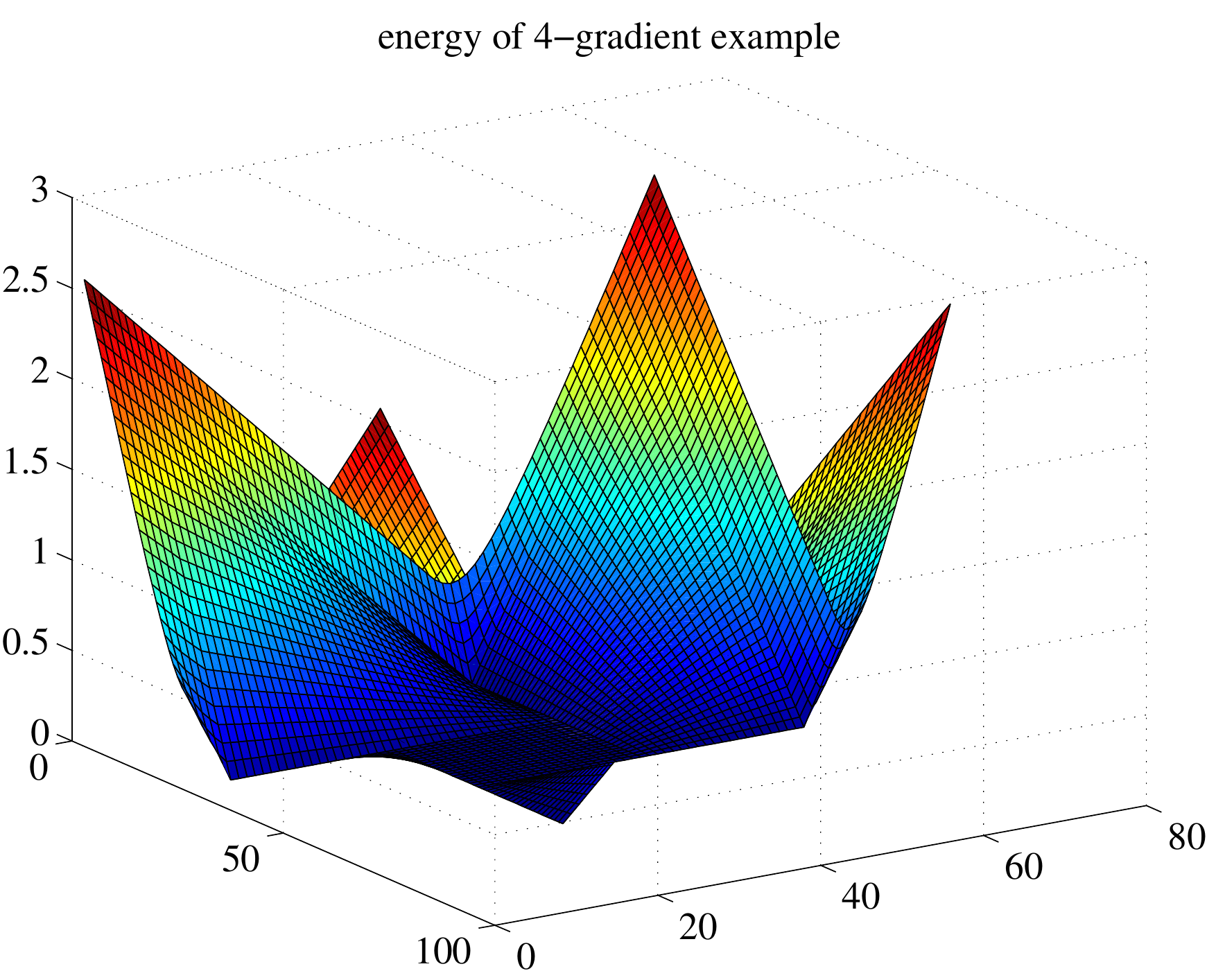}
\includegraphics{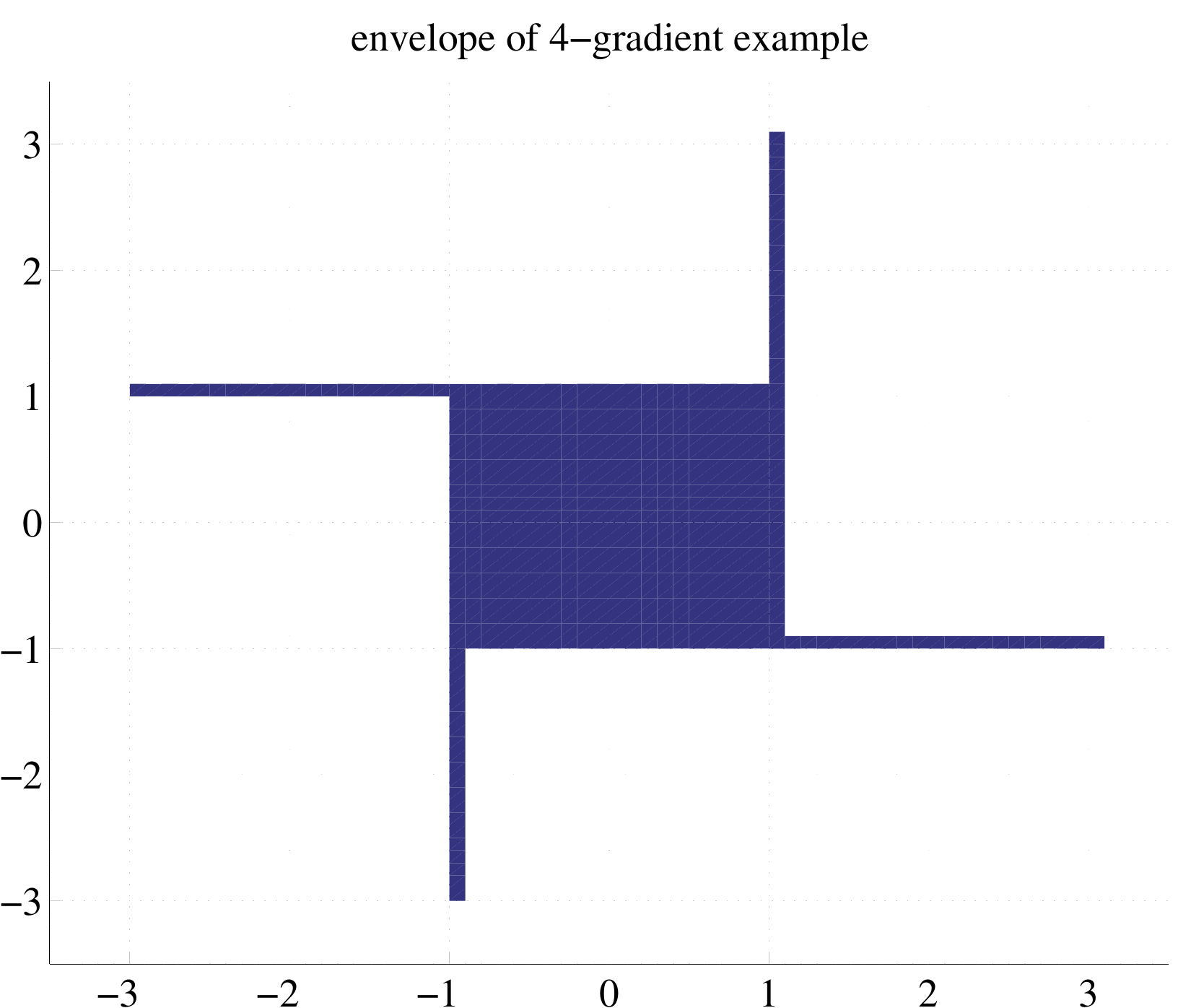}
}
\scalebox{.25}{
\hspace{-4cm}
\includegraphics{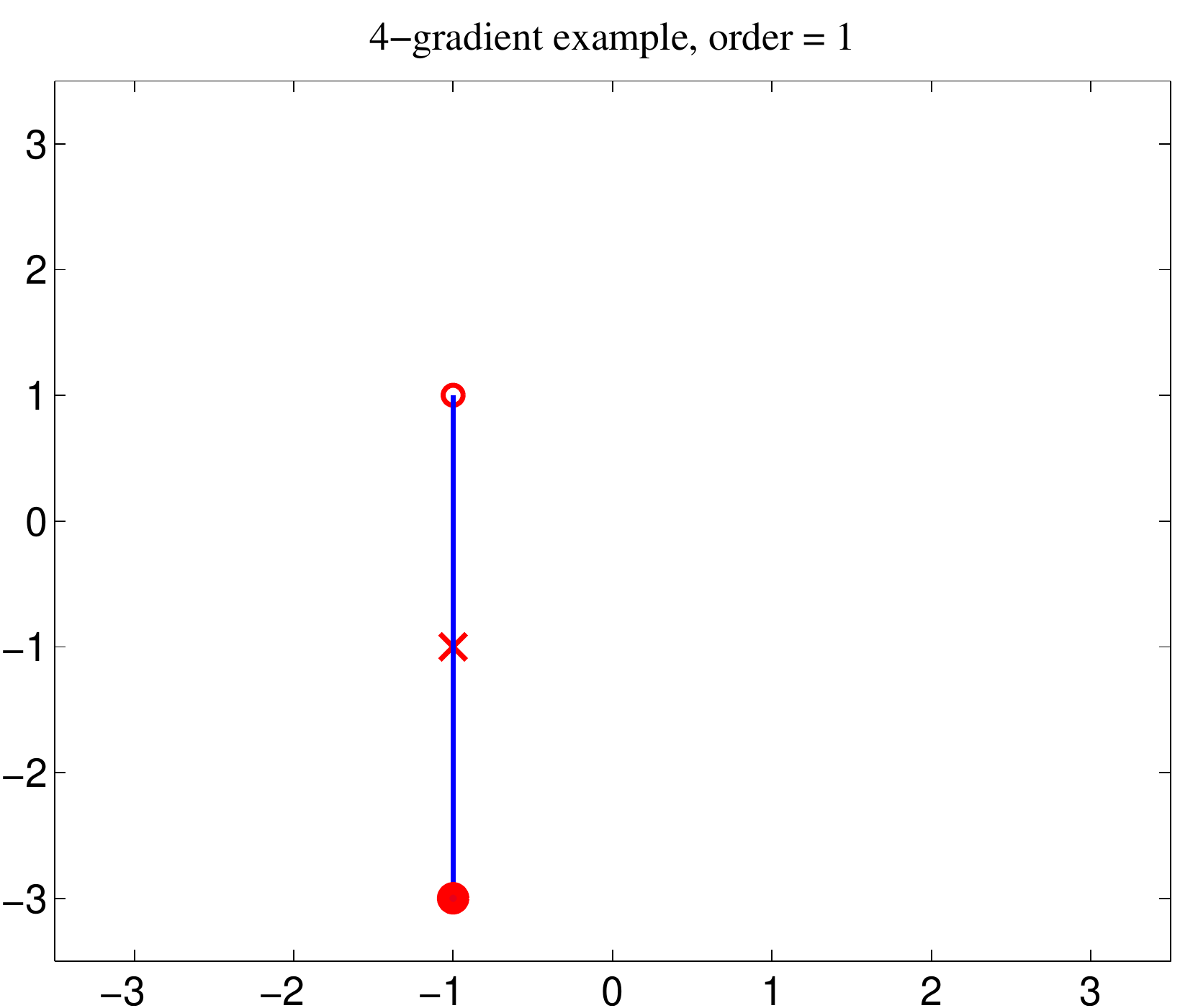}
\includegraphics{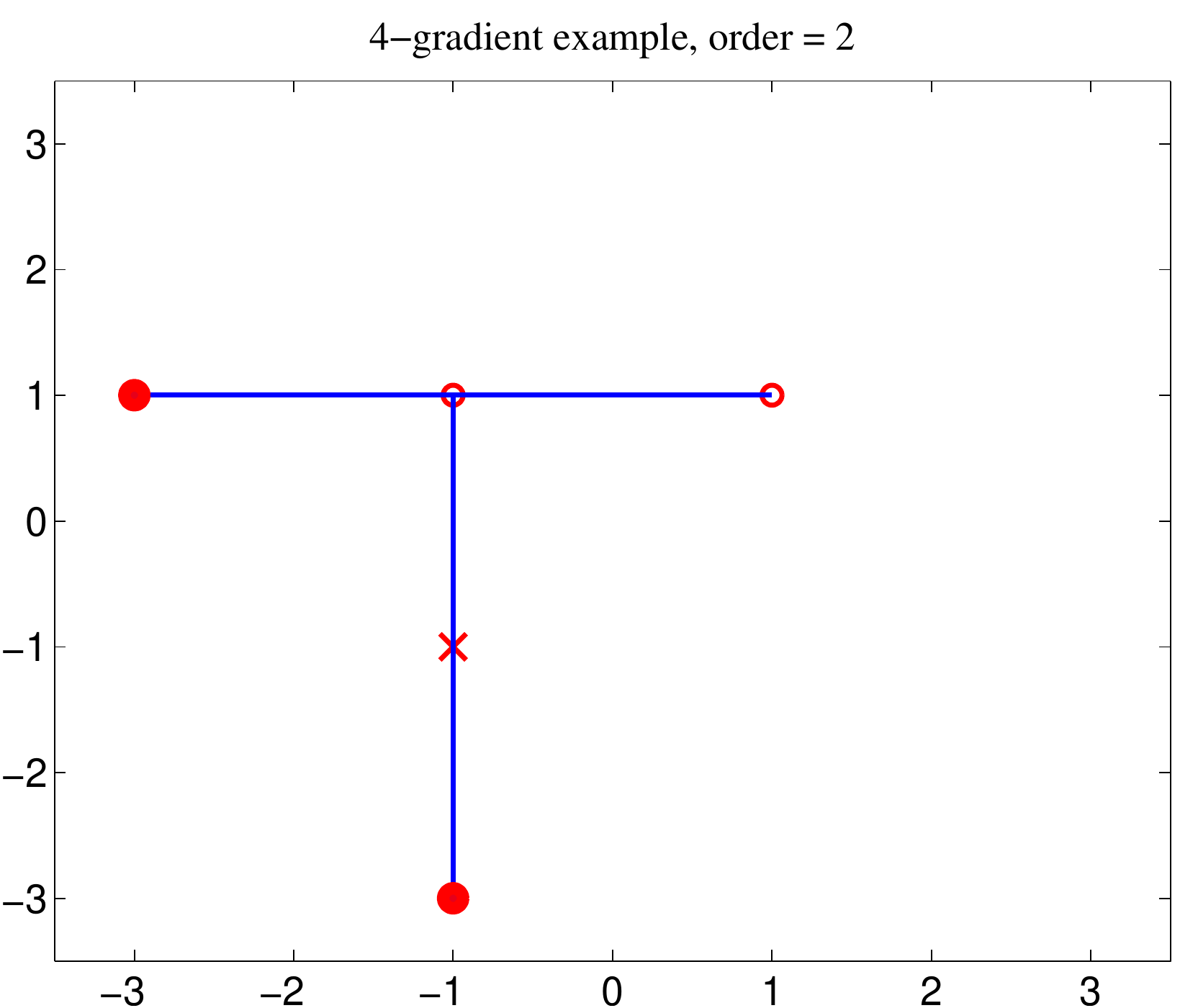}
\includegraphics{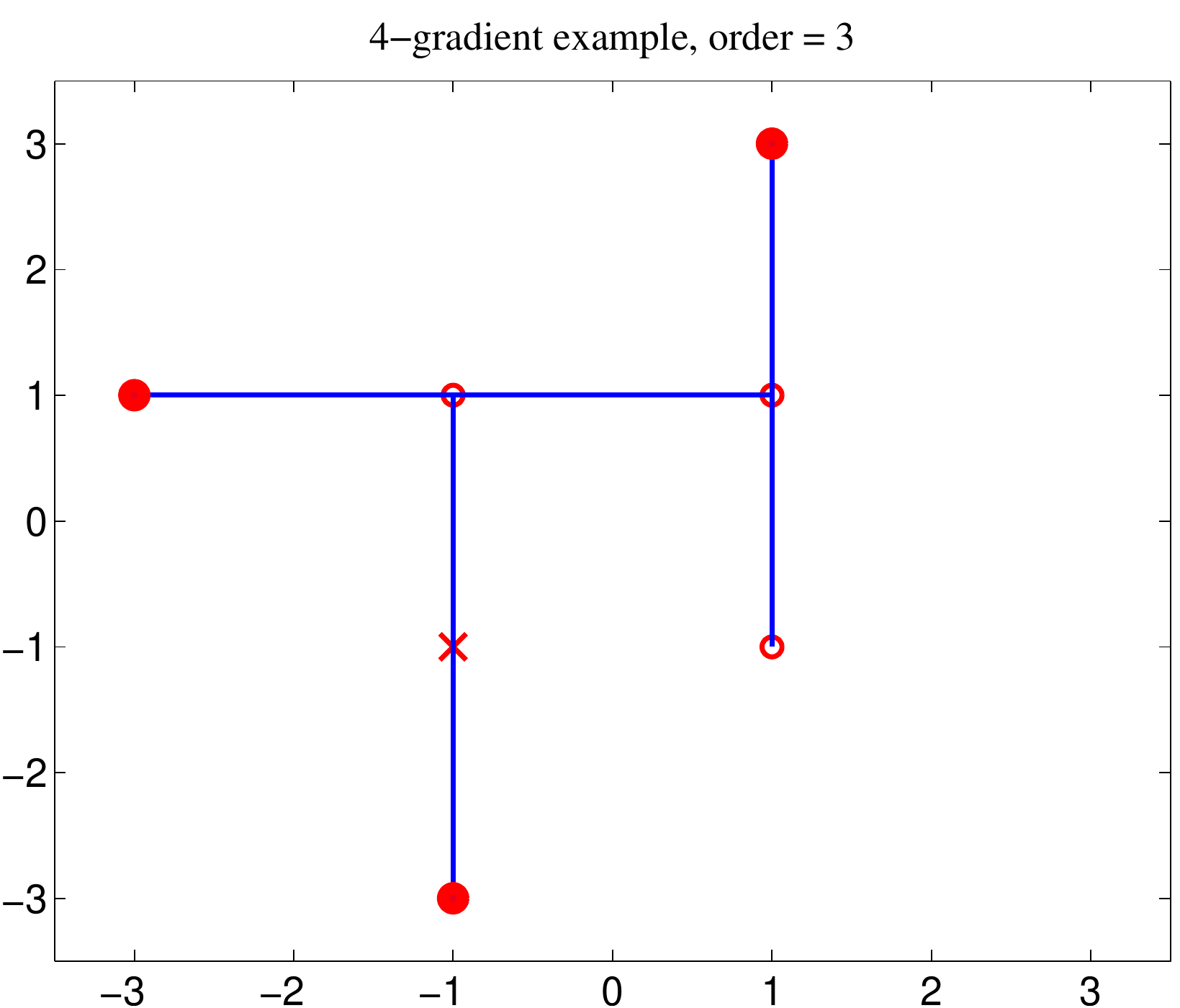}
\includegraphics{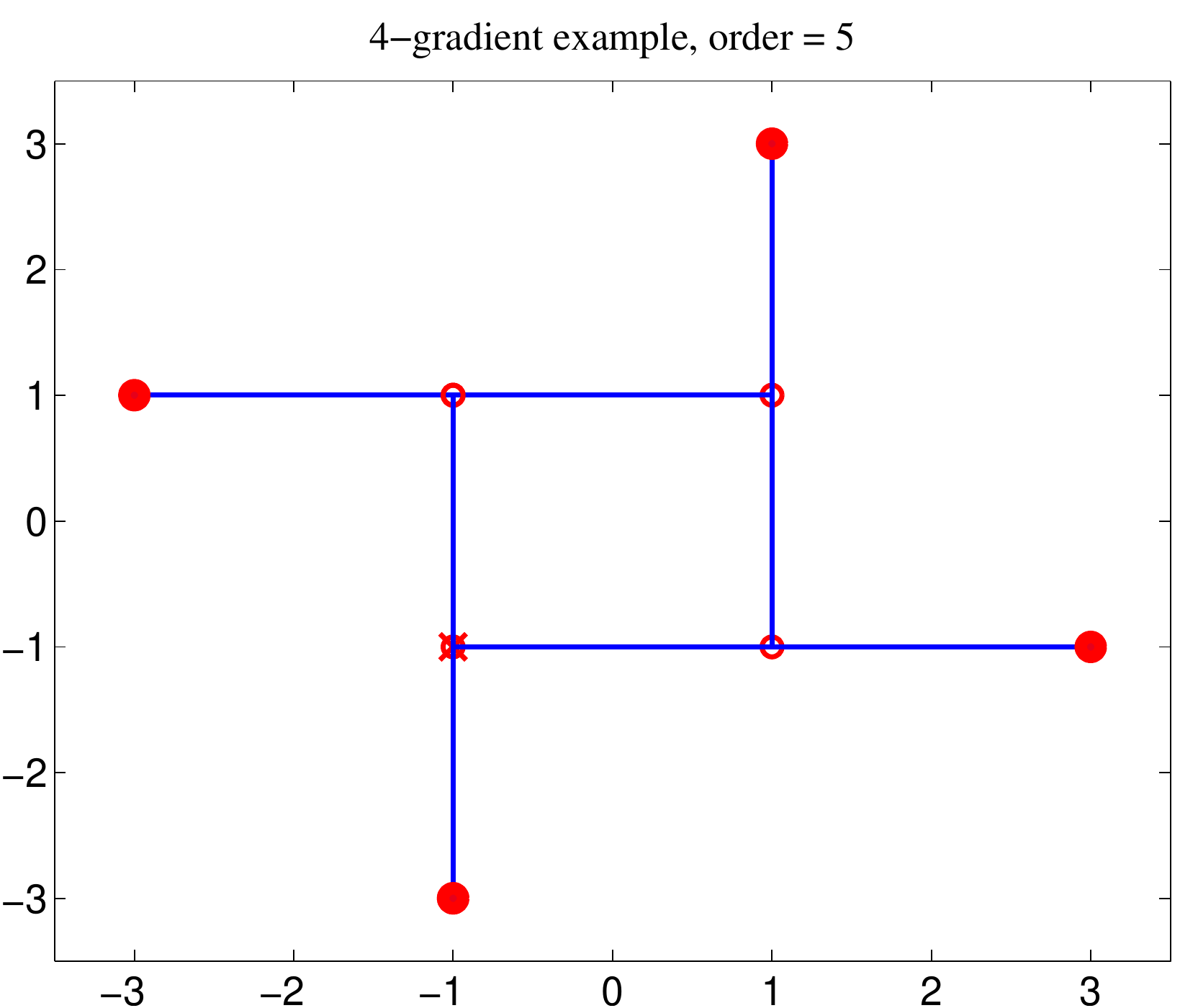}
}

\caption{Example~\ref{ex:ClassicalFourGradient}. 
Top: Rank one convex envelope and hull. 
Bottom: laminates extracted directly from the computed rank one convex envelope. The levels increase from left to right. 
}%
\label{fig:4Ghull}%
\label{fig:lam2d}
\label{fig:lam2dmulti}%
\end{figure}

\subsection{A synthetic four gradient example}\label{ex:SyntheticFourGradient}
To illustrate the $\Dir$-convex envelope, we construct the following synthetic example, which is easier to visualize than the higher dimensional examples which follow.

\begin{example}
Consider again the set $K$ and the function $G$,  as in \eqref{K4g} of in Example~\ref{eg4gradient}.  Set
\[
\Dir_4 = \{ (1,0), (0,1), (1,1), (-1,1) \}
\]
The computed $\Dir_4$-convex hull is shown in Figure~\ref{fig:4Ghullsyn} below.
The shape of the directional convex hull is predictable. 
The $\Dir_4$-convex hull has a hexagonal
shape, and contains the $\Dir_2$-convex hull from Example~\ref{ex:ClassicalFourGradient}.
Figure~\ref{fig:lam2dsyn} shows how the laminate is computed by constructing
the $(H)$-sequence.
\end{example}

\begin{figure}[h]
\scalebox{.37}{
\hspace{-3cm}
\includegraphics{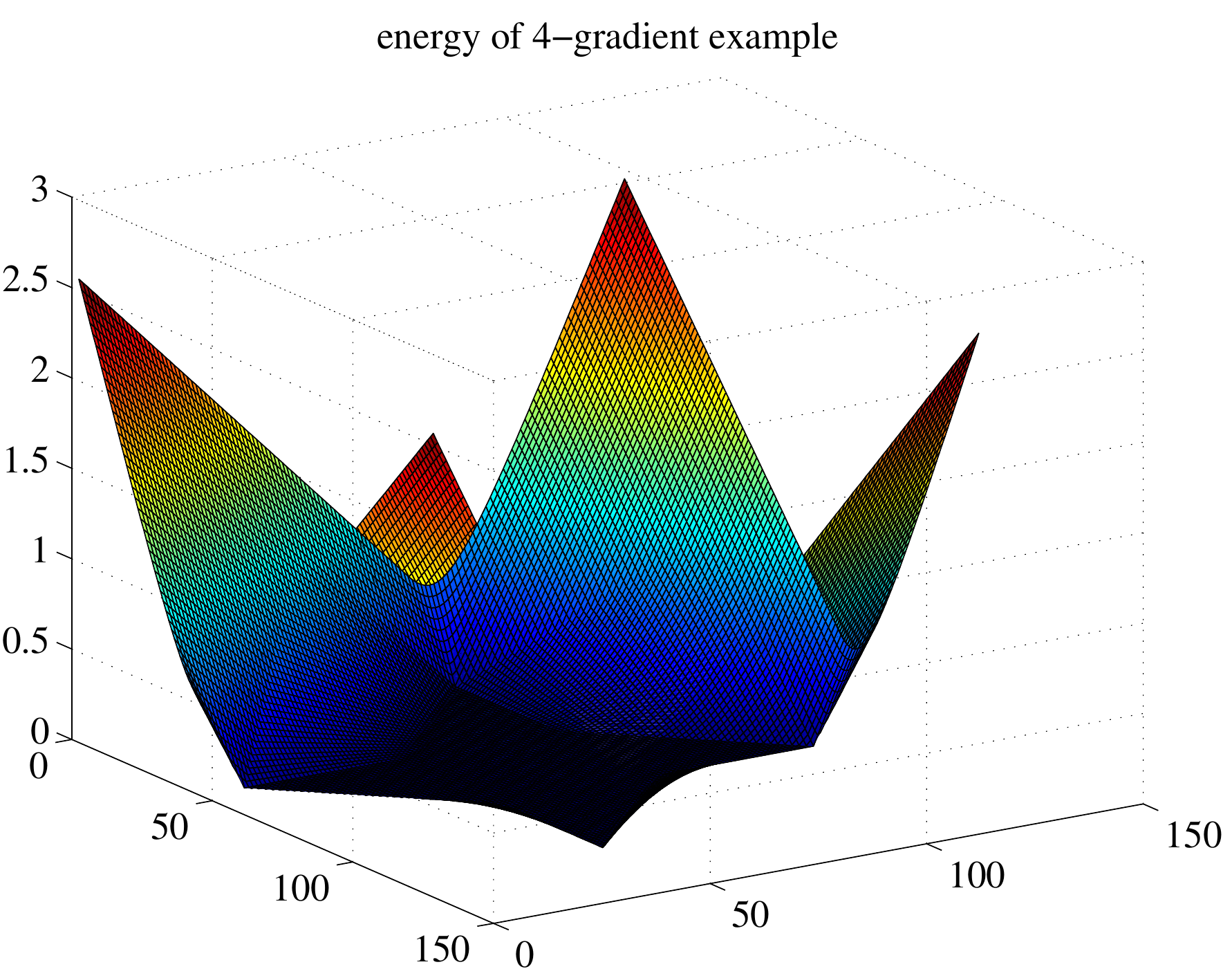}
\includegraphics{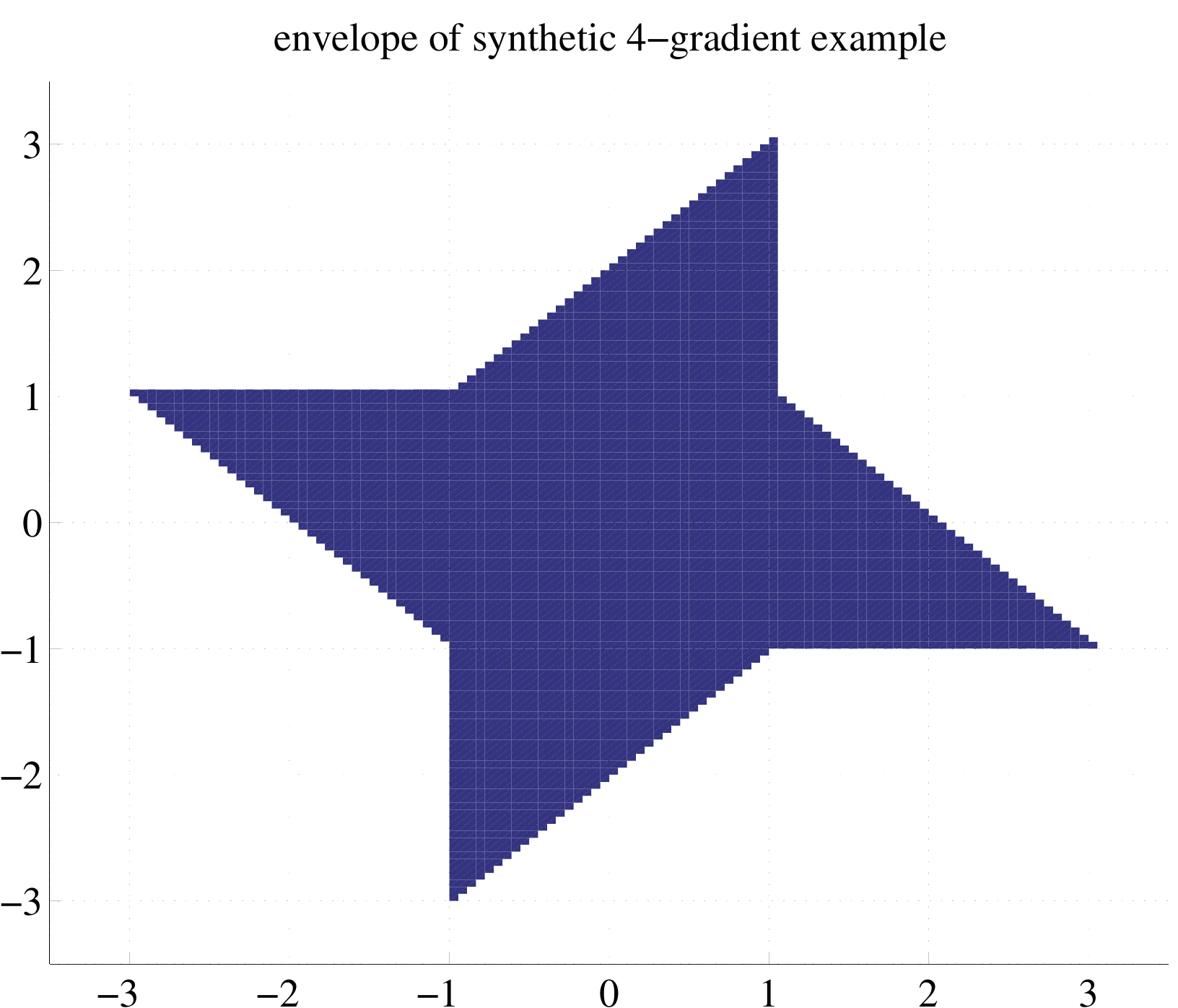}
}
\scalebox{.2}{
\includegraphics{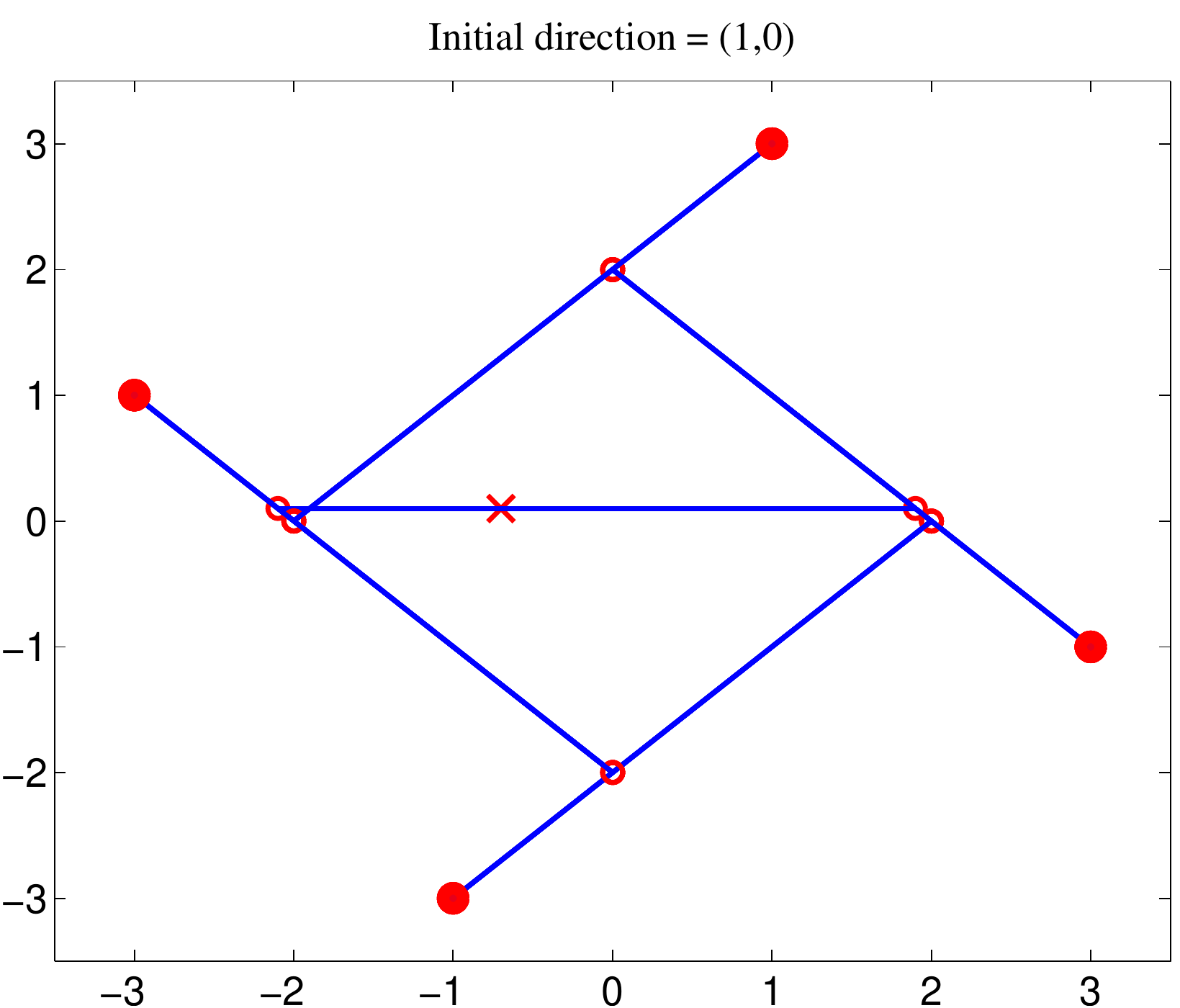}
\includegraphics{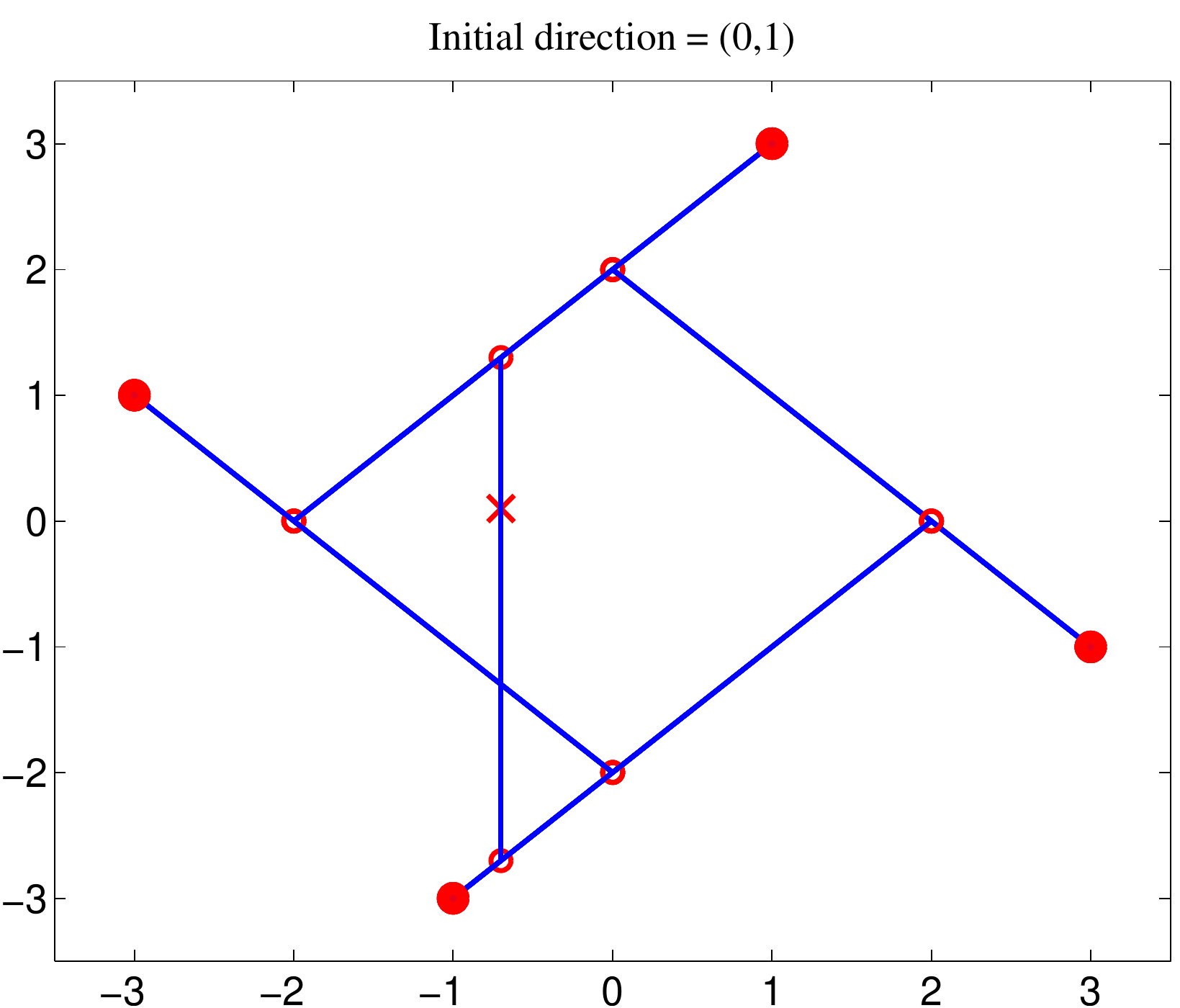}
\includegraphics{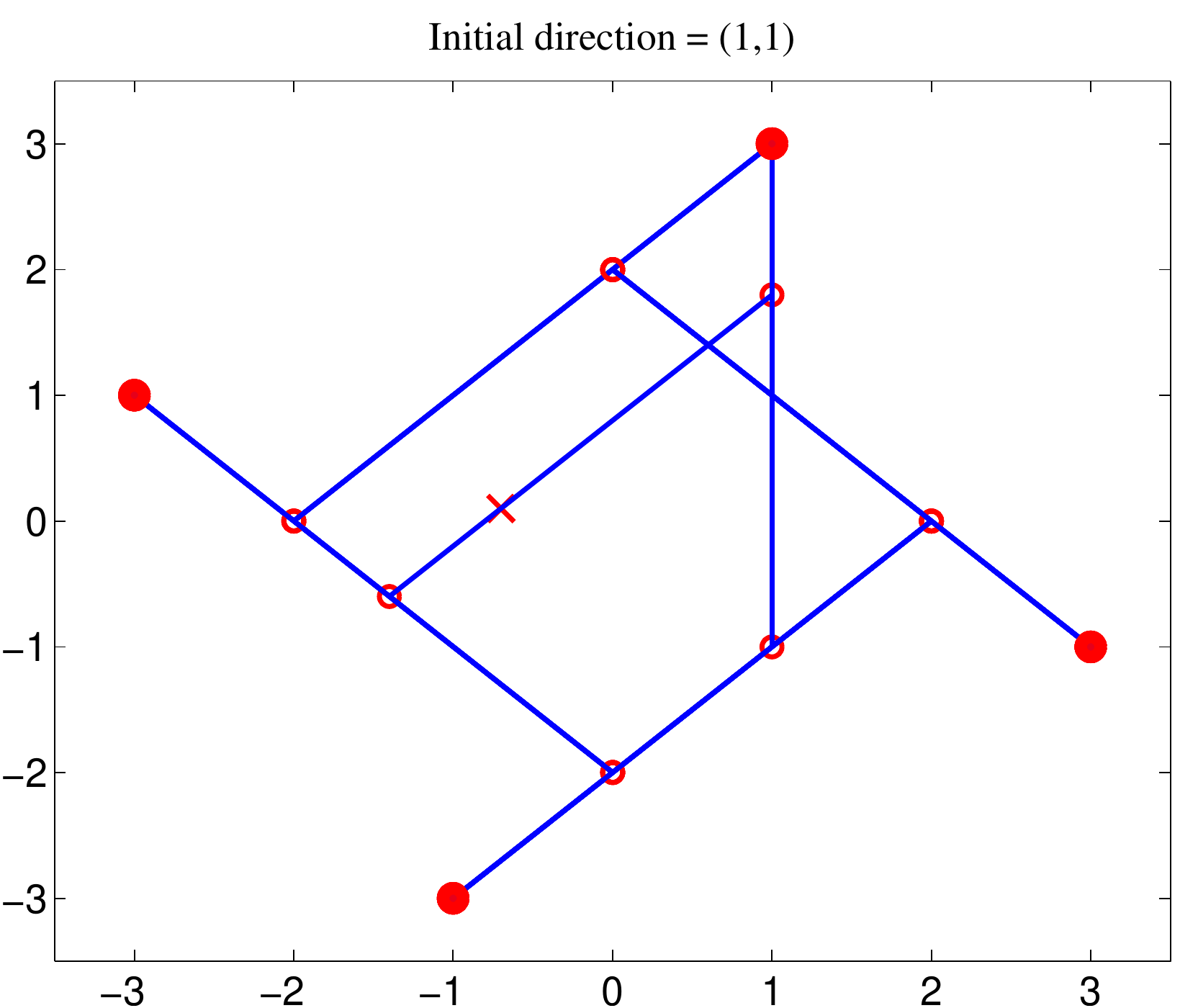}
\includegraphics{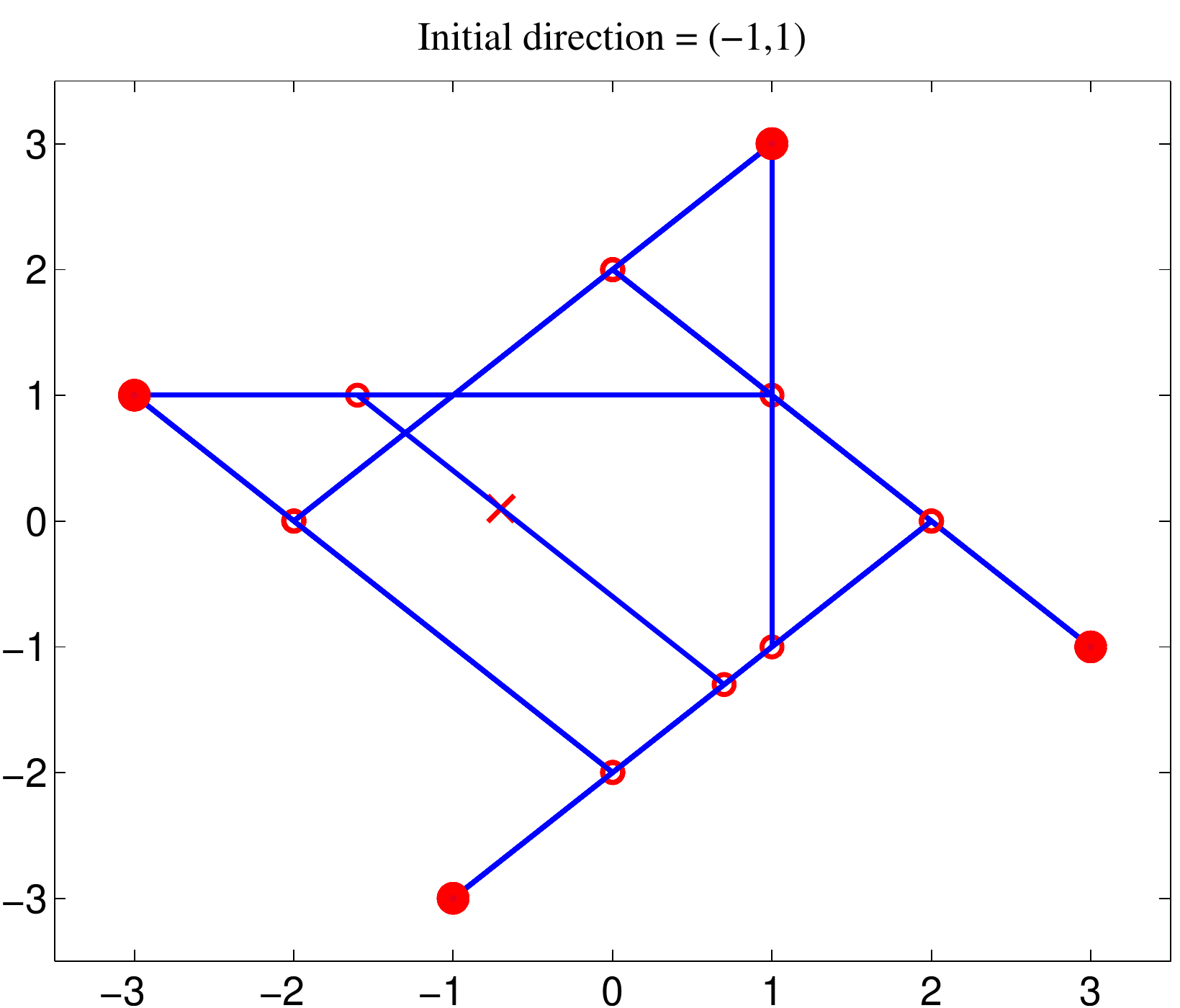}
}
\caption{Top left: $\Dir_4$-convex envelope, Top right: $\Dir_4$ convex hull of four points.
Bottom: different laminates with the same barycenter, generated using different initial directions.
}%
\label{fig:4Ghullsyn}%
\label{fig:lam2dsyn}%
\end{figure}

\subsection{Computation times and accuracy for the two dimensional examples}
In this section we present convergence results and solution times for the two dimensional examples. 

Table~\ref{table:2dlevelset} shows the convergence of the area of the computed convex hull for the two examples, as a function of the grid resolution.  
In Table~\ref{table:solutiontimes2d} we compare the solution time using two different methods: the function iteration and convexification along lines.  The maximum error tolerance was $10^{-8}$.   In this case (with only two directions) the latter method is faster.  In cases with more directions the opposite occurs.

\begin{table}[]
\centering
\begin{tabular}{llllllll}
Gridsize & $28^2$ & $42^2$ & $56^2$ & $70^2$ & $84^2$ & $98^2$ & $112^2$ \\\hline
Classic example & 7.0625 & 6.0278 & 5.5156 & 5.2100 & 5.0069 & 4.8622 & 4.7539 \\
Synthetic example & 14.063 & 13.361 & 13.016 & 12.810 & 12.674 & 12.577 & 12.504
\end{tabular}
\caption{Convergence of the area of the zero level set in terms of grid size for Examples~\ref{ex:ClassicalFourGradient} and \ref{ex:SyntheticFourGradient}. 
}
\label{table:2dlevelset}
\end{table}

\begin{table}[pth]
\begin{tabular}
[c]{ccccc}%
N &  CPU Time (LS)  & Iterations (LS) & CPU Time (IS) & Iterations (IS)          \\\hline
43 & 0.52 & 17 & 0.72 & 839\\
71 & 1.51 & 17& 3.19 & 2257\\
127 & 5.02 & 18& 22.49 & 7036
\end{tabular}

\begin{tabular}
[c]{ccccc}%
N &  CPU Time (LS)  & Iterations (LS) & CPU Time (IS) & Iterations (IS)          \\\hline
43 & 0.63 & 11  & 0.5 & 398\\
71 & 1.66 & 11  & 2.3 & 1065\\
127 & 5.12 & 11 & 17.3 & 3330
\end{tabular}

\caption{Computation time of the rank one convex envelope for Example~\ref{ex:ClassicalFourGradient}.
Comparing the convexification along line solver (LS) with the explicit iterative solver (IS).
 $N$ is the number of points in each dimension. 
 Bottom: corresponding table for Example~\ref{ex:SyntheticFourGradient}.}%
\label{table:solutiontimes2d}%
\end{table}

\subsection{A three dimensional example}
Next we consider a synthetic three dimensional example. 
\begin{example}
Consider the set $K_6=\{A_{1},\dots ,A_{6}\}$, where the first four entries are given by \eqref{K4g} from 
Example~\ref{eg4gradient} and
\[
A_{5}=-A_{6}=\left(
\begin{array}
[c]{cc}%
0 & 3\\
0 & 0
\end{array}
\right)  .
\]
The six matrices occupy only three entries of the 2 by 2 matrices, so we regard it as a synthetic three dimensional problem. The corresponding directional convex envelope and laminates are computed below.
\end{example}

\begin{figure}[h]
\scalebox{.37}{
\hspace{-3cm}
\includegraphics{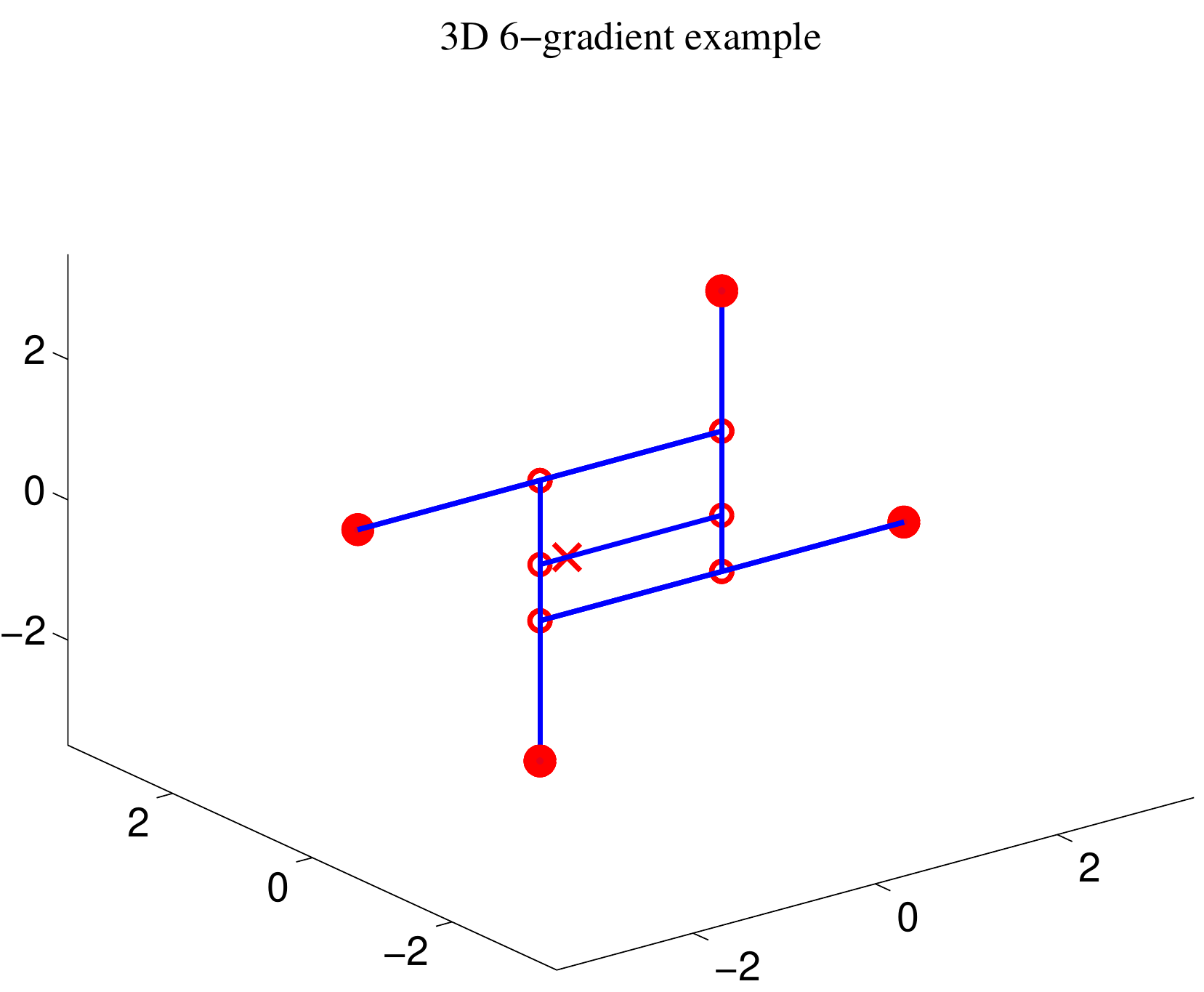}
\includegraphics{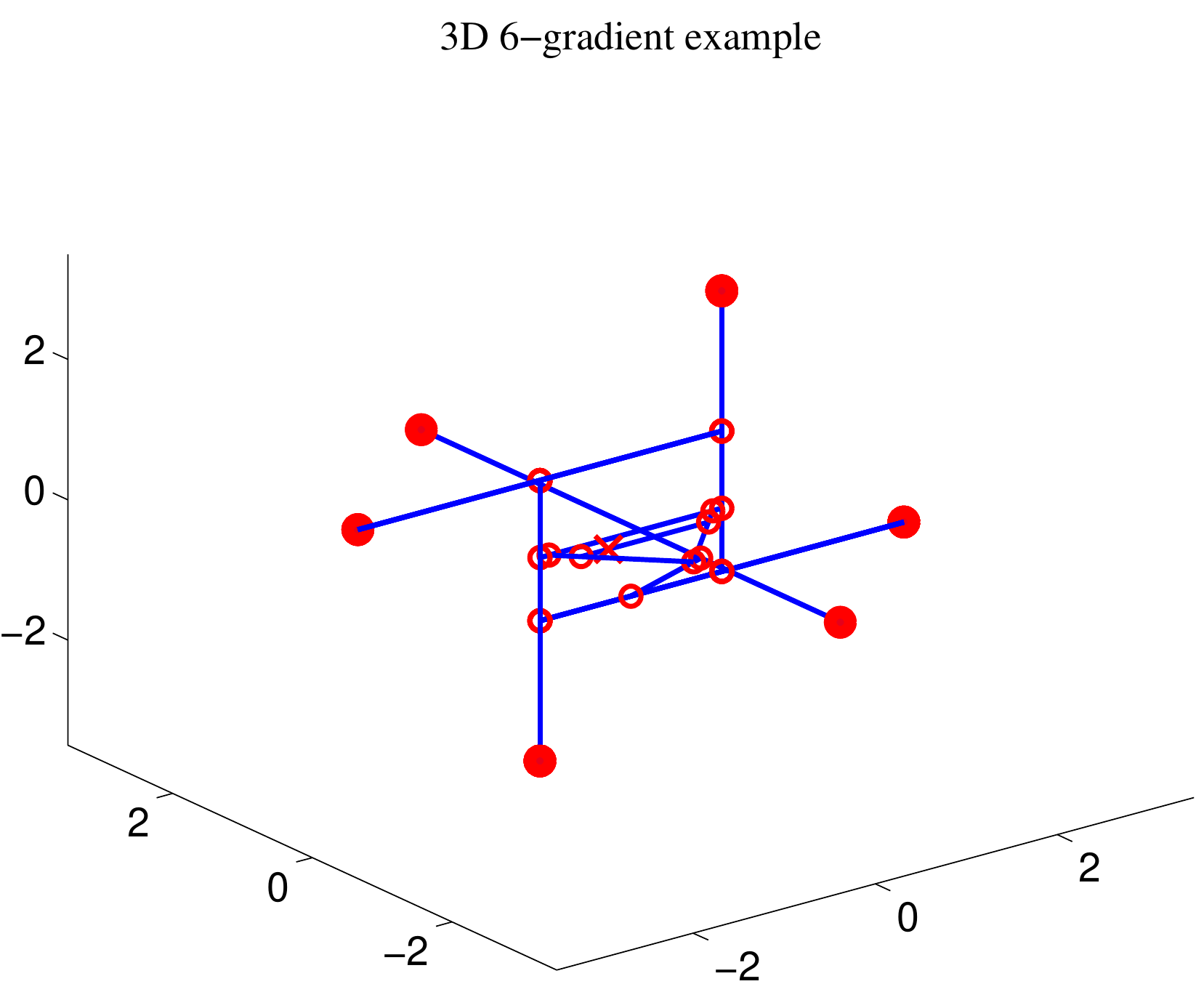}
}\caption{Laminate for the six gradient problem.}%
\label{fig:Lam6G}%
\end{figure}

Since the set $K_6$ falls on the subspace spanned by%
\[
\left\{  \left(
\begin{array}
[c]{cc}%
x & y\\
0 & z
\end{array}
\right)  :x,y,z\in\mathbb{R}\right\}  ,
\]
we consider rank one directions%
\[
\Dir_{7}=\left\{
\begin{array}
[c]{c}%
\left(
\begin{array}
[c]{cc}%
1 & 0\\
0 & 0
\end{array}
\right)  ,\left(
\begin{array}
[c]{cc}%
0 & 1\\
0 & 0
\end{array}
\right)  ,\left(
\begin{array}
[c]{cc}%
0 & 0\\
0 & 1
\end{array}
\right)  ,\left(
\begin{array}
[c]{cc}%
1 & 1\\
0 & 0
\end{array}
\right)  ,\\
\left(
\begin{array}
[c]{cc}%
0 & 1\\
0 & 1
\end{array}
\right)  ,\left(
\begin{array}
[c]{cc}%
1 & -1\\
0 & 0
\end{array}
\right)  ,\left(
\begin{array}
[c]{cc}%
0 & 1\\
0 & -1
\end{array}
\right)
\end{array}
\right\}  .
\]

Figure~\ref{fig:Lam6G} shows two laminates. The first one has its starting
point on the plane spanned by $\{A_{1},A_{2},A_{3},A_{4}\}$ so it resembles the
classical example, while the starting point for the second one is not. Below
we list a few sample laminates with the same barycenter as the second graph of
Figure~\ref{fig:Lam6G}. 

Let $\upsilon^{k}$ denote the laminate generated with initial decomposition direction
being the $k$-th vector in $\Dir_{7}$ (ordered as shown above). Write 
$\bar \upsilon^k =  \langle \upsilon^k,1\rangle_{K_6}$, for the concentration on
supporting set $K_6$.  This quantity is a measure of the accuracy of the approximation. We find after a few iterations,
\begin{align*}
\upsilon^{1} &=\left(0.238332,0.251663,0.168331,0.141665,0.066667,0.133333\right), & \bar \upsilon^{1}=0.999991
\\
\upsilon^{5}&=\left(
0.218326,0.231663,0.148330,0.121664,0.106667,0.173333\right), & \bar \upsilon^{5} =0.999984
\\
\upsilon^{6}&=\left(
0.207077,0.220413,0.137080,0.110415,0.129167,0.195833\right), & \bar \upsilon^{6} =0.999985.
\end{align*}

\subsection{Another three dimensional example}
We next consider an example which is described in~\cite[p171]%
{pedregal1997parametrized}.  In this problem, the gradients are identified with the three dimensional subspace of the form
\[
M(x,y,z) = 
\left(
\begin{array}
[c]{cc}%
x+z & z\\
z & y+z
\end{array}
\right)
\]
and the corresponding rank one directions are given by the set of directions which satisfy $\det(M) = 0$, $\Dir = \{xy+yz+xz=0\}.$ 

Consider the function $f(x,y,z)=xyz$ defined on the cube $\left[-1,1\right]^{3}$
with rank one directions contained in $\Dir$. 

For the computation, the direction vectors used, $\Dir_{24}$ consists of the vectors   $\{(1,0,0), (-1,2,2), (-2,3,6), (-12,3,4), (-6,10,15)\}$ and their permutations. These vectors were generated by taking two small integers and solving the equation $z = -xy/(x+y)$ for the third one. For example $(x,y) = (1,2)$ gives $z = -2/3$ and multiplying by $3$ gives the vector $(3,6,-2)$.  Since the stencils are wide, we needed to pad the domain by the appropriate amount.   Note that the density of direction vectors appears to be low for this example.  We extended the grid to account for the wide stencil, and we used a cutoff function which was a difference of exponentials in each coordinate to enforce \eqref{g_assumption} on the extended part of the grid.

The approximate solutions were computed using $\Dir_{24}$ and an interior grid size (neglecting the padding) of $21^3$ and $31^3$.  Solution values at the origin were  $-0.49786$ and $-0.50000$ for the smaller and larger grid, respectively.  These values are close to the known analytical value of $-1/2$.

\subsection{A four dimensional eight gradient problem}
\begin{example}\label{ex:8gradient}
Consider the set $K_8 =\{A_{1},\dots, A_{8}\}$, with the first four entries given by \eqref{K4g} from 
Example~\ref{eg4gradient} and
\[
A_{5}=-A_{7}=\left(
\begin{array}
[c]{cc}%
0 & -2\\
-1 & 0
\end{array}
\right)  ,\text{ }A_{6}=-A_{8}=\left(
\begin{array}
[c]{cc}%
0 & 1\\
-2 & 0
\end{array}
\right).
\]
No rank one connections exist in $K_8.$ 
For this eight-gradient problem, the visualization of the laminates is more difficult.    The example we computed here gives computational evidence for the existence of minimizers which are not nearly affine as proven in~\cite[Theorem 7.12]{Dacorogna2}.

As a test of consistency, we recover the laminates from Example~\ref{eg4gradient}, by taking a barycenter 
on the plane spanned by $\{A_{1},\dots,A_{4}\}.$  This is pictured in Figure~\ref{fig:8GLam371} top, 
which shows projection onto two planes of the laminate.
For general barycenters, the rank one convex hull has a more complex structure.  Figure \ref{fig:8GLam371} shows the laminate with barycenter in general position.   One example of laminate with the same barycenter as figure \ref{fig:8GLam371} is given by 
\[
\upsilon=(
0.275264,0.092443,0.041928,0.225774,0.207919,0.023986,0.007921,0.123975)
\]
with $\bar \upsilon^{1} =0.999212.$

We measured the  convergence of the volume of the zero level set, in 
 Table~\ref{table:4dlevelset}.  The increase in the volume going from $\Dir_{64}$ to $\Dir_{144}$ is significant, which shows the need for higher directional resolution.  However the change from $\Dir_{144}$ to  $\Dir_{256}$ is much smaller, which suggests convergence for this example.  Likewise, the volume is not changing much as a function of $\dx$.  The change in values in the middle column may just be an artifact of the grid, compared to the locations of the points of $K_8$.  
\end{example}

\begin{figure}[h]
\scalebox{.3}{
\hspace{-3cm}
\includegraphics{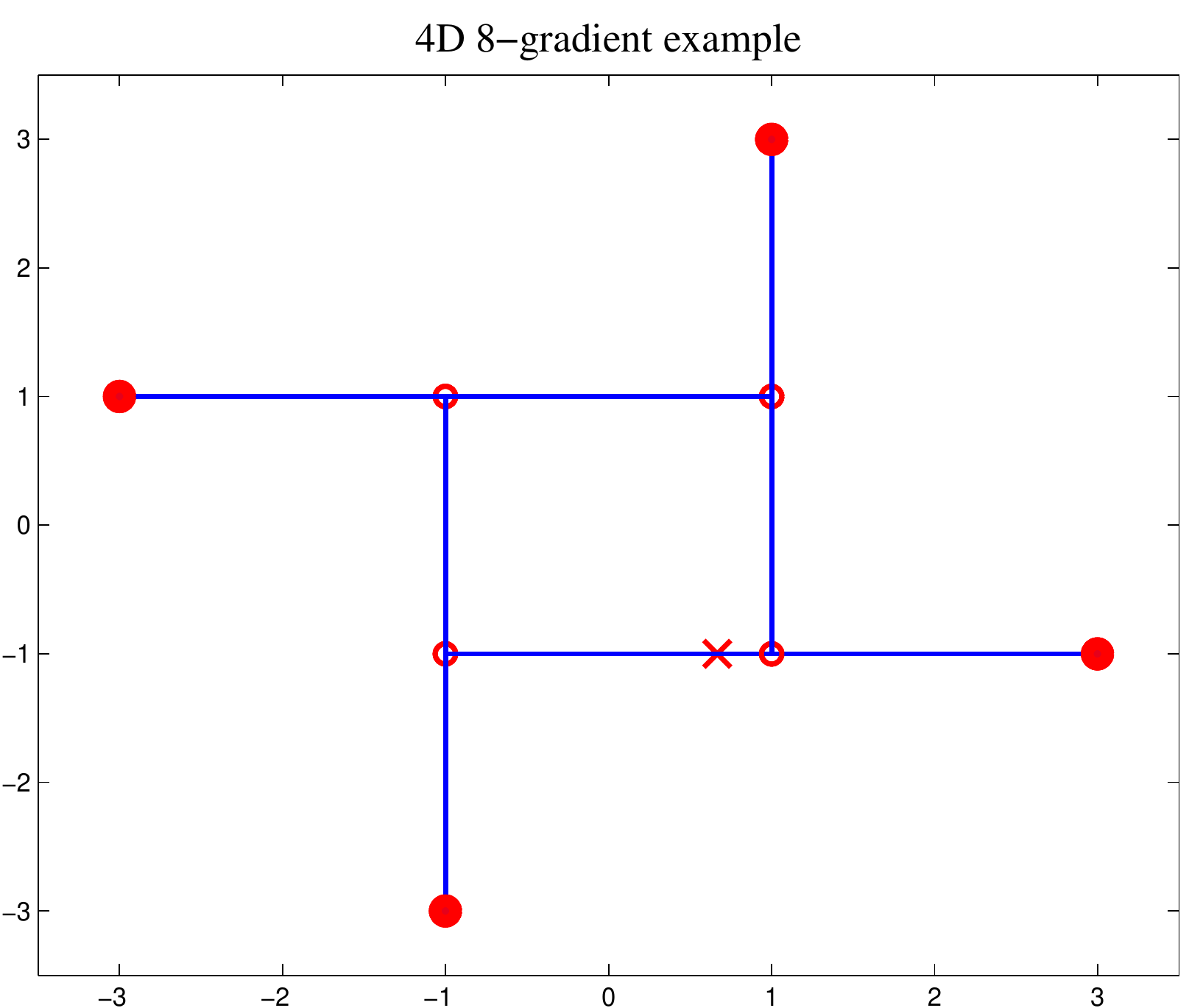}
\includegraphics{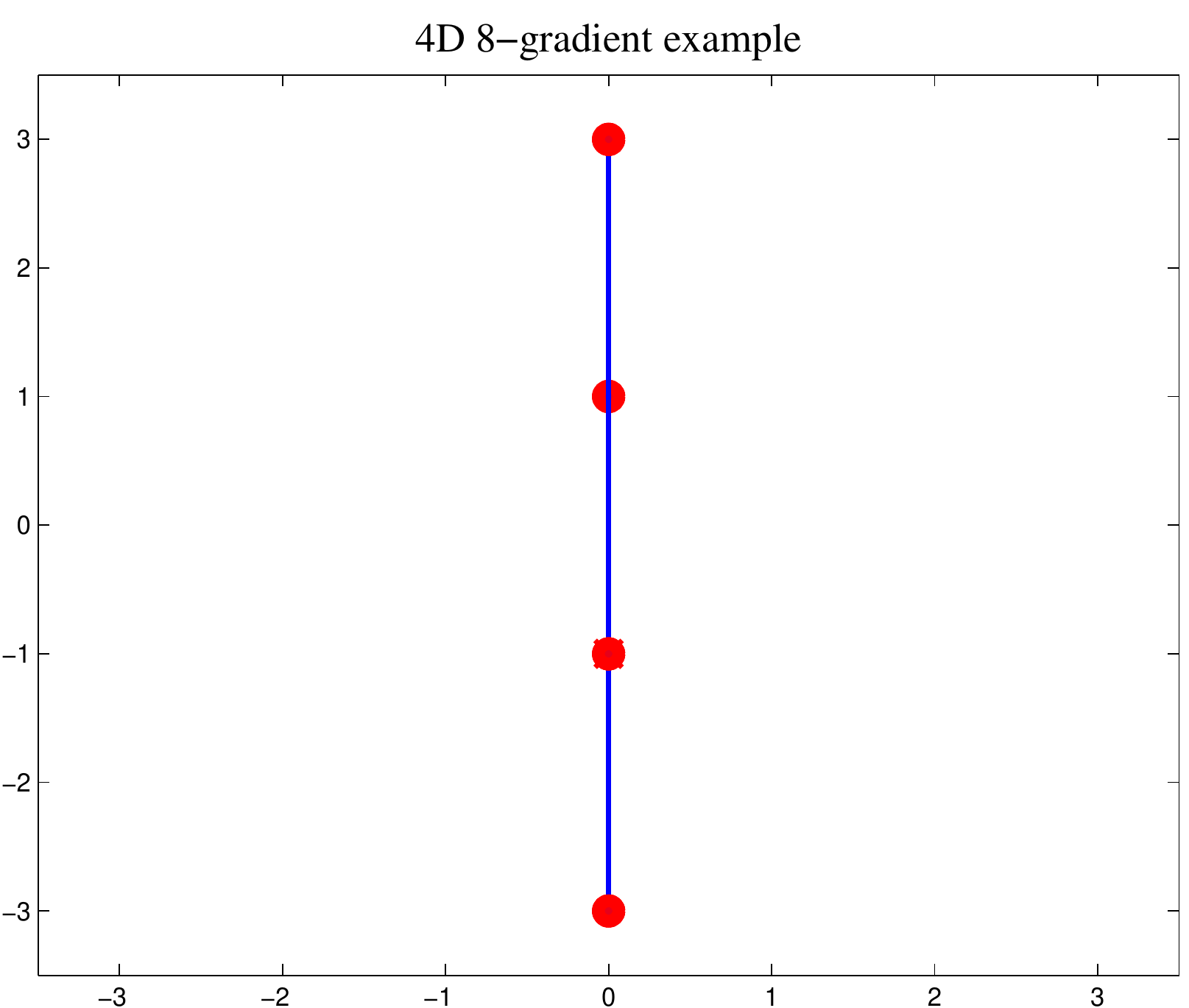}
}
\scalebox{.3}{
\hspace{-3cm}
\includegraphics{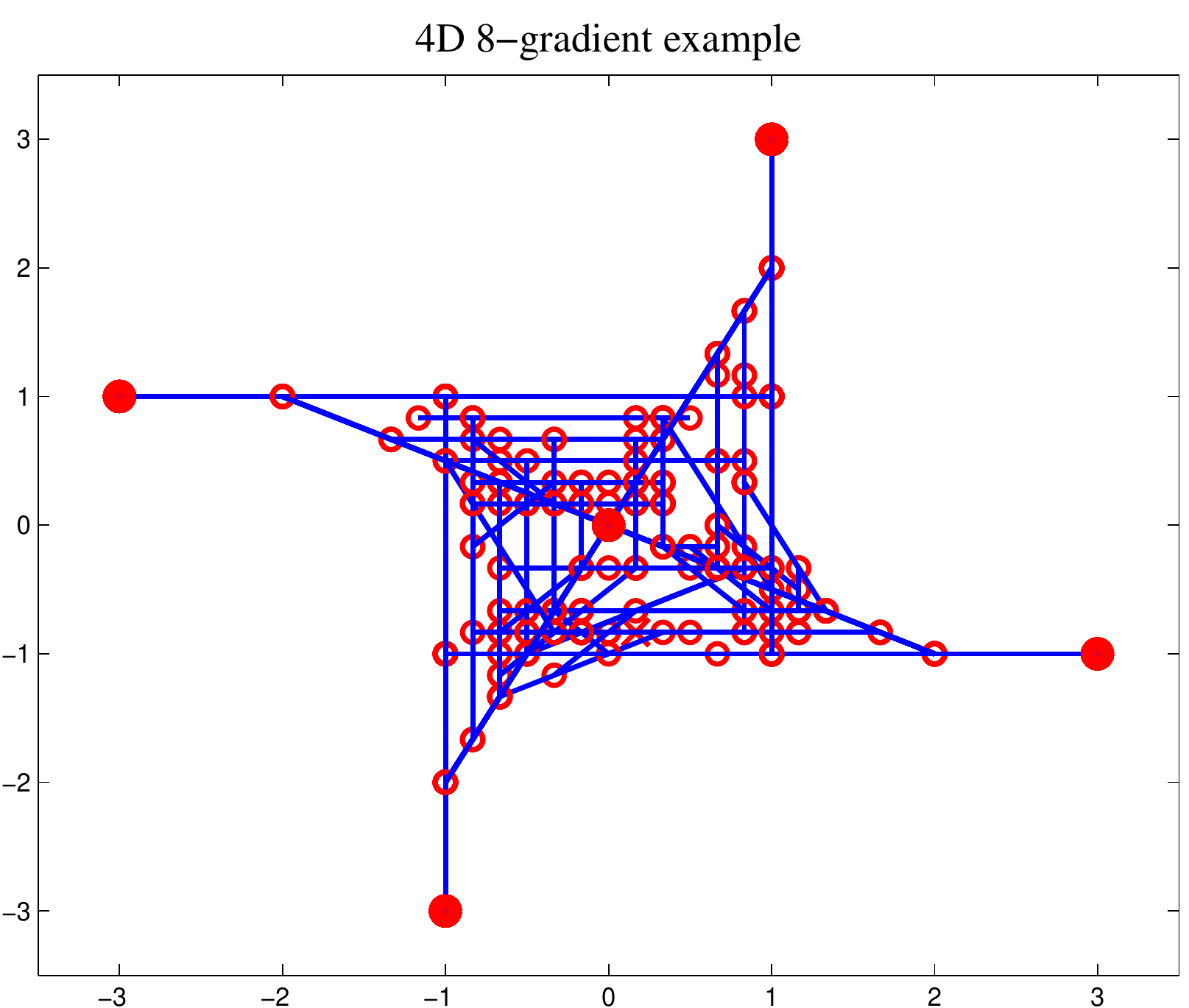}
\includegraphics{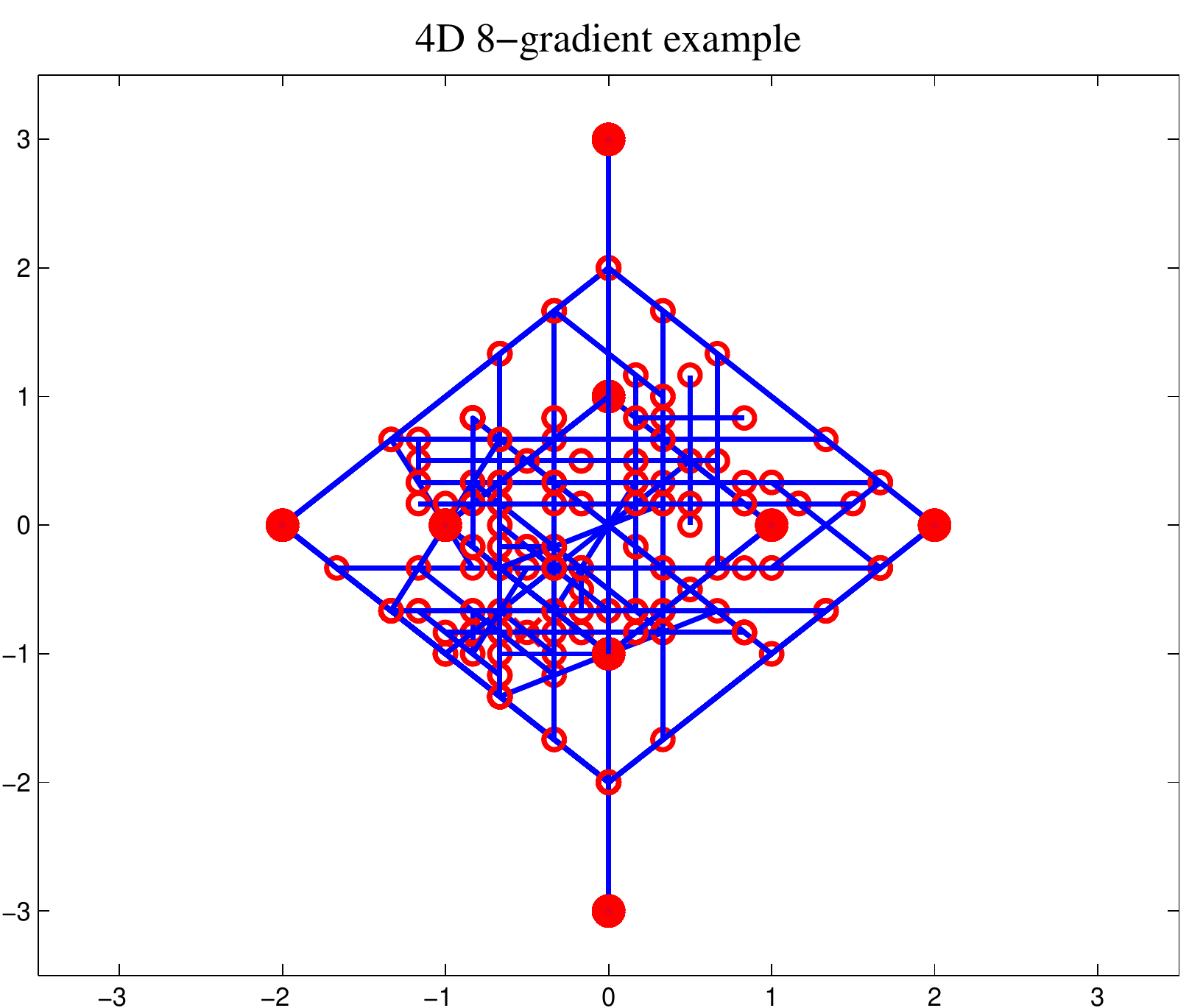}
}\caption{Projection onto the $x-w$ and $y-w$ plane of laminates from the four dimensional example. 
Top: barycentre lies on the $x-w$ plane. Bottom: barycentre in general position.}%
\label{fig:8GLam371}%
\end{figure}

%

\begin{table}[ptb]
\centering
\begin{tabular}
[c]{ccllll}%
Gidsize & $dx$ & $\Dir_{16}$ & $\Dir_{64}$ & $\Dir_{144}$ & $\Dir_{256}$ \\\hline
$45^4$ & 0.2500 & 2.2227 & 6.5325 & 27.254 & 27.316\\
$57^4$ & 0.1667 & 1.5934 & 7.3773 & 24.606 & 25.396\\
$69^4$ & 0.1250 & 1.3792 & 6.7815 & 27.256 & 27.715
\end{tabular}
\caption{Convergence of the volume of the zero level set in terms of rank one directions for Example~\ref{ex:8gradient}. }%
\label{table:4dlevelset}
\end{table}

\section{Conclusions}
We introduced a nonlinear degenerate elliptic partial differential equation in the form of an obstacle problem for the rank one convex envelope (and more generally, for directional convex envelopes). 
The PDE is consistent and well-posed: there exist unique viscosity solutions and these solutions give the rank one convex envelope of the obstacle function. Existence of solutions continuous up to the boundary was established using Perron's method.

  A convergent finite difference scheme was presented: we showed that there exist unique solutions of the discrete equation, and that these solutions can be computed by a simple iterative method.   An iterative algorithm was presented for obtaining laminates from the computed envelopes. 

We computed several numerical examples in two, three and four spatial dimensions.  These included the classical Kohn-Strang example and the classical four matrix example.  We also computed laminates for two, three and four dimensional problems, including an eight gradient example, which gave rise to complex laminates.  In particular, we gave computational evidence for \cite[Theorem 7.12]{Dacorogna2} which states the existence of minimizers which are not nearly affine.

We studied the computational speed and the accuracy of the method, in terms of the spatial and directional resolution.  The largest problem, in four dimensions, used 25 million variables, and 256 directions.  This problem was computed in about half an hour on a recent model laptop.

Another possible extension would be to increase the accuracy of the method using a filtered scheme~\cite{froese2013convergent} to approximate directional derivatives in off grid directions.  An improvement in solution speed could be obtained by alternating the directional line solver with the iterative method as in \cite{ObermanQuasicConvex}.

We are limited by problem size from computing directly the three by three matrix case, since it gives rise to a nine dimensional problem. This is too large for our method which requires representing the function on a spatial grid.  However, special cases of the three by three case which lead to lower dimensional problems could be computed.   

There are other types of direction sets and other PDEs resulting from rank one convexity which we hope to study in the future. For example, rotation invariant rank one convex functions satisfy relations which could lead to a different PDE~\cite[Chapter 5]{Dacorogna2}.



\bibliographystyle{alpha}
\bibliography{R1CE_Jan_09}

\begin{thebibliography}{Obe08b}

\bibitem[ALL97]{alvarez1997convex}
Olivier Alvarez, J-M Lasry, and P-L Lions.
\newblock Convex viscosity solutions and state constraints.
\newblock {\em Journal de math{\'e}matiques pures et appliqu{\'e}es},
  76(3):265--288, 1997.

\bibitem[AO16]{ObermanQuasicConvex}
Bilal Abbasi and Adam~M Oberman.
\newblock A partial differential equation for the strictly quasiconvex
  envelope.
\newblock {\em arXiv:1612.06813}, 2016.

\bibitem[AP01a]{aranda2001numerical}
Ernesto Aranda and Pablo Pedregal.
\newblock Numerical approximation of non-homogeneous, non-convex vector
  variational problems.
\newblock {\em Numerische Mathematik}, 89(3):425--444, 2001.

\bibitem[AP01b]{pedregal2001computation}
Ernesto Aranda and Pablo Pedregal.
\newblock On the computation of the rank-one convex hull of a function.
\newblock {\em SIAM Journal on Scientific Computing}, 22(5):1772--1790, 2001.

\bibitem[Bal77]{BallElasticity}
John~M. Ball.
\newblock Convexity conditions and existence theorems in nonlinear elasticity.
\newblock {\em Arch. Rational Mech. Anal.}, 63(4):337--403, 1976/77.

\bibitem[Bar04]{bartels2004linear}
S{\"o}ren Bartels.
\newblock Linear convergence in the approximation of rank-one convex envelopes.
\newblock {\em ESAIM: Mathematical Modelling and Numerical Analysis},
  38(05):811--820, 2004.

\bibitem[Bar05]{bartels2005reliable}
S{\"o}ren Bartels.
\newblock Reliable and efficient approximation of polyconvex envelopes.
\newblock {\em SIAM Journal on Numerical Analysis}, 43(1):363--385, 2005.

\bibitem[BJ89]{BallJames87}
John Ball and Richard James.
\newblock Fine phase mixtures as minimizers of energy.
\newblock In {\em Analysis and Continuum Mechanics}, pages 647--686. Springer,
  Berlin, 1989.

\bibitem[BKK00]{ball2000regularity}
John~M Ball, Bernd Kirchheim, and Jan Kristensen.
\newblock Regularity of quasiconvex envelopes.
\newblock {\em Calculus of Variations and Partial Differential Equations},
  11(4):333--359, 2000.

\bibitem[BM06]{Bardi2006709}
Martino Bardi and Paola Mannucci.
\newblock On the {D}irichlet problem for non-totally degenerate fully nonlinear
  elliptic equations.
\newblock {\em Communications on Pure and Applied Analysis}, 5(4):709--731,
  2006.

\bibitem[BM13]{bardi2013comparison}
Martino Bardi and Paola Mannucci.
\newblock Comparison principles and {D}irichlet problem for fully nonlinear
  degenerate equations of {M}onge--{A}mp{\`e}re type.
\newblock In {\em Forum Mathematicum}, volume~25, pages 1291--1330, 2013.

\bibitem[BS91]{BSNum}
Guy Barles and Panagiotis~E. Souganidis.
\newblock Convergence of approximation schemes for fully nonlinear second order
  equations.
\newblock {\em Asymptotic Anal.}, 4(3):271--283, 1991.

\bibitem[CIL92]{CIL}
Michael~G. Crandall, Hitoshi Ishii, and Pierre-Louis Lions.
\newblock User's guide to viscosity solutions of second order partial
  differential equations.
\newblock {\em Bull. Amer. Math. Soc. (N.S.)}, 27(1):1--67, 1992.

\bibitem[CK88]{chipot1988equilibrium}
Michel Chipot and David Kinderlehrer.
\newblock Equilibrium configurations of crystals.
\newblock {\em Archive for Rational Mechanics and Analysis}, 103(3):237--277,
  1988.

\bibitem[CNS86]{caffarelli1986dirichlet}
Luis~A Caffarelli, Louis Nirenberg, and Joel Spruck.
\newblock The dirichlet problem for the degenerate monge-amp{\`e}re equation.
\newblock {\em Revista Matem{\'a}tica Iberoamericana}, 2(1-2):19--27, 1986.

\bibitem[Dac08]{Dacorogna2}
Bernard Dacorogna.
\newblock {\em Direct methods in the calculus of variations}, volume~78 of {\em
  Applied Mathematical Sciences}.
\newblock Springer, Berlin, second edition, 2008.

\bibitem[Dol99]{Dolzmann}
Georg Dolzmann.
\newblock Numerical computation of rank-one convex envelopes.
\newblock {\em SIAM J. Numer. Anal.}, 36(5):1621--1635 (electronic), 1999.

\bibitem[Dol03]{dolzmann2003variational}
Georg Dolzmann.
\newblock {\em Variational methods for crystalline microstructure-analysis and
  computation}.
\newblock Number 1803. Springer Science \& Business Media, 2003.

\bibitem[DPF15]{de2015optimal}
Guido De~Philippis and Alessio Figalli.
\newblock Optimal regularity of the convex envelope.
\newblock {\em Transactions of the American Mathematical Society},
  367(6):4407--4422, 2015.

\bibitem[DW00]{DolzmannWalkington}
G.~Dolzmann and N.~J. Walkington.
\newblock Estimates for numerical approximations of rank one convex envelopes.
\newblock {\em Numer. Math.}, 85(4):647--663, 2000.

\bibitem[FM09]{franvek2009computing}
Vojt{\v{e}}ch Fran{\v{e}}k and Ji{\v{r}}{\'\i} Matou{\v{s}}ek.
\newblock Computing d-convex hulls in the plane.
\newblock {\em Computational Geometry}, 42(1):81--89, 2009.

\bibitem[FO13]{froese2013convergent}
Brittany~D Froese and Adam~M Oberman.
\newblock Convergent filtered schemes for the {M}onge--{A}mp{\`e}re partial
  differential equation.
\newblock {\em SIAM Journal on Numerical Analysis}, 51(1):423--444, 2013.

\bibitem[Fro16]{froeseGauss}
Brittany~D Froese.
\newblock Convergent approximation of surfaces of prescribed {G}aussian
  curvature with weak {D}irichlet conditions.
\newblock {\em arXiv:1601.06315}, 2016.

\bibitem[GT83]{GTBook}
David Gilbarg and Neil~S. Trudinger.
\newblock {\em Elliptic partial differential equations of second order}, volume
  224 of {\em Grundlehren der Mathematischen Wissenschaften [Fundamental
  Principles of Mathematical Sciences]}.
\newblock Springer-Verlag, Berlin, second edition, 1983.

\bibitem[KS86a]{kohn1986optimal1}
Robert~V Kohn and Gilbert Strang.
\newblock Optimal design and relaxation of variational problems, i.
\newblock {\em Communications on Pure and Applied Mathematics}, 39(1):113--137,
  1986.

\bibitem[KS86b]{kohn1986optimal2}
Robert~V Kohn and Gilbert Strang.
\newblock Optimal design and relaxation of variational problems, ii.
\newblock {\em Communications on Pure and Applied Mathematics}, 39(2):139--182,
  1986.

\bibitem[Mor52]{morrey1952quasi}
Charles~B Morrey.
\newblock Quasi-convexity and the lower semicontinuity of multiple integrals.
\newblock {\em Pacific J. Math}, 2(1):25--53, 1952.

\bibitem[MP98]{matouvsek1998functional}
Jir̆{\'\i} Matou{\v{s}}ek and P~Plech{\'a}{\v{c}}.
\newblock On functional separately convex hulls.
\newblock {\em Discrete \& Computational Geometry}, 19(1):105--130, 1998.

\bibitem[Mul99]{Muller}
Stefan Muller.
\newblock Variational models for microstructure and phase transitions.
\newblock In {\em Calculus of Variations and Geometric Evolution Problems
  (Italy, 1996)}, pages 85--210. Springer, Berlin, 1999.

\bibitem[MW53]{MotzkinWasow}
Theodore~S. Motzkin and Wolfgang Wasow.
\newblock On the approximation of linear elliptic differential equations by
  difference equations with positive coefficients.
\newblock {\em J. Math. Physics}, 31:253--259, 1953.

\bibitem[Obe06]{ObermanSINUM}
Adam~M. Oberman.
\newblock Convergent difference schemes for degenerate elliptic and parabolic
  equations: {H}amilton-{J}acobi equations and free boundary problems.
\newblock {\em SIAM J. Numer. Anal.}, 44(2):879--895 (electronic), 2006.

\bibitem[Obe07]{ObermanConvexEnvelope}
Adam~M. Oberman.
\newblock The convex envelope is the solution of a nonlinear obstacle problem.
\newblock {\em Proc. Amer. Math. Soc.}, 135(6):1689--1694 (electronic), 2007.

\bibitem[Obe08a]{ObermanCENumerics}
Adam~M. Oberman.
\newblock Computing the convex envelope using a nonlinear partial differential
  equation.
\newblock {\em Math. Models Methods Appl. Sci.}, 18(5):759--780, 2008.

\bibitem[Obe08b]{ObermanEigenvalues}
Adam~M. Oberman.
\newblock Wide stencil finite difference schemes for the elliptic
  {M}onge-{A}mp\`ere equation and functions of the eigenvalues of the
  {H}essian.
\newblock {\em Discrete Contin. Dyn. Syst. Ser. B}, 10(1):221--238, 2008.

\bibitem[OS11]{oberman2011dirichlet}
Adam Oberman and Luis Silvestre.
\newblock The {D}irichlet problem for the convex envelope.
\newblock {\em Transactions of the American Mathematical Society},
  363(11):5871--5886, 2011.

\bibitem[Ped97]{pedregal1997parametrized}
Pablo Pedregal.
\newblock {\em Parametrized measures and variational principles}, volume~30.
\newblock Springer, 1997.

\end{thebibliography}
\end{document}